\documentclass[12pt]{amsart}

\usepackage{amsmath,amssymb,latexsym,amsthm,newlfont,enumerate}

\usepackage{graphicx}
\usepackage[all]{xy}
\input diagxy

\theoremstyle{plain}

\newtheorem{thm}{Theorem}[section]

\newtheorem{pro}[thm]{Proposition}

\newtheorem{lem}[thm]{Lemma}

\newtheorem{applem}[thm]{Approximation Lemma}
\newtheorem{cla}[thm]{Claim}
\newtheorem{cor}[thm]{Corollary}
\newtheorem{con}[thm]{Conjecture}

\newtheorem{hmck}[thm]{Lifting Lemma}

\theoremstyle{definition}

\newtheorem{dfn}[thm]{Definition}

\newtheorem{nt}[thm]{Notation}

\newtheorem{dfnlm}[thm]{Definition-Lemma}
\newtheorem{rem}[thm]{Remark}

\newtheorem{exa}[thm]{Example}

\newtheorem{stp}[thm]{Setup}
\theoremstyle{remark}

\DeclareMathOperator{\Sing}{Sing}

\DeclareMathOperator{\nklt}{nklt}

\DeclareMathOperator{\lct}{lct}

\DeclareMathOperator{\mult}{mult}
\DeclareMathOperator{\res}{res}
\DeclareMathOperator{\im}{Im}

\DeclareMathOperator{\Supp}{Supp}
\DeclareMathOperator{\Bs}{Bs}
\DeclareMathOperator{\Diff}{Diff}
\DeclareMathOperator{\Div}{Div}
\DeclareMathOperator{\WDiv}{WDiv}
\DeclareMathOperator{\ddiv}{div}
\DeclareMathOperator{\codim}{codim}
\DeclareMathOperator{\Exc}{Exc}

\DeclareMathOperator{\bDiv}{\mathbf{Div}}
\DeclareMathOperator{\bdiv}{\mathbf{div}}
\DeclareMathOperator{\Mob}{Mob}

\DeclareMathOperator{\Fix}{Fix}

\DeclareMathOperator{\Proj}{Proj}

\newcommand{\R}{\mathbb{R}}
\newcommand{\Q}{\mathbb{Q}}
\newcommand{\N}{\mathbb{N}}
\newcommand{\Z}{\mathbb{Z}}

\newcommand{\D}{\mathbf{D}}
\newcommand{\C}{\mathbb{C}}

\newcommand{\FF}{\mathbf{F}}
\newcommand{\M}{\mathbf{M}}
\newcommand{\NN}{\mathbf{N}}
\newcommand{\A}{\mathbf{A}}
\newcommand{\B}{\mathbf{B}}
\newcommand{\K}{\mathbf{K}}

\newcommand{\DDelta}{\mathbf{\Delta}}
\newcommand{\NE}{\mathrm{NE}}

\newcommand{\hcf}{\mathrm{hcf}}
\newcommand{\Image}{\mathrm{Image}}
\newcommand{\OO}{\mathcal{O}}

\newcommand{\Pic}{\mathrm{Pic}}

\newcommand{\mbb}{\mathbb}
\newcommand{\mcal}{\mathcal}

\title[Introduction to the MMP and the existence of flips]{Introduction to the Minimal Model Program and the existence of flips}
\author{Alessio Corti}
\address{Department of Mathematics, Imperial College London, Huxley Bu\-il\-ding, 180 Queen's Gate, London SW7 2AZ, United
Kingdom}
\email{A.Corti@imperial.ac.uk}

\author{Anne-Sophie Kaloghiros}
\address{Department of Pure Mathematics and Mathematical Statistics, Uni\-ver\-si\-ty of Cambridge, Wilberforce Road, Cambridge CB3 0WB, Uni\-ted
Kingdom}
\email{A.S.Kaloghiros@dpmms.cam.ac.uk}

\author{Vladimir Lazi\'c}
\address{Department of Pure Mathematics and Mathematical Statistics, Uni\-ver\-si\-ty of Cambridge, Wilberforce Road, Cambridge CB3 0WB, Uni\-ted
Kingdom}
\email{V.Lazic@dpmms.cam.ac.uk}

\begin{document}

\begin{abstract}
The first aim of this note is to give a concise, but complete and self-contained, presentation of the fundamental theorems of Mori
theory---the nonvanishing, base point free, rationality and cone theorems---using modern methods of multiplier ideals, Nadel vanishing,
and the subadjunction theorem of Kawamata. The second aim is to write up a complete, detailed proof of existence of
flips in dimension $n$ assuming the minimal model program with scaling in dimension $n-1$.
\end{abstract}

\maketitle
\bibliographystyle{amsalpha}

\tableofcontents
\section{Introduction}

Our first aim in this note is to give a concise, but complete and
self-contained, presentation of the fundamental theorems of Mori
theory---the nonvanishing, base point free, rationality and cone
theorems---using modern methods of multiplier ideals, Nadel vanishing,
and the subadjunction theorem of Kawamata. We also give the basic
definitions of log terminal and log canonical singularities of pairs,
and spell out the minimal model program. We hope that the text can
be used as a fast introduction to the field for those who wish quickly
to master the foundations and equip themselves to do research. The
approach here is not, at heart, different from the traditional
one, but it is more efficient and it allows one to focus attention on
the basic issues without being distracted by the technicalities. It
also serves as a demonstration of the power of multiplier ideals, which
play a crucial role in the most recent advances in the field, and the
subadjunction theorem, which is a beautiful and increasingly central
result.

We briefly sketch the key point as it arises in the proof of
the base point free theorem. Let $X$ be a nonsingular projective
variety over $\mathbb{C}$, and $L$ a nef divisor on $X$. Assuming that
$L-\varepsilon K$ is ample for all small $\varepsilon >0$, we want to
show that all large multiples $nL$ are base point free. The first step
is to construct an effective $\mathbb{Q}$-divisor $D\sim mL$ (for some
$m$) such that the pair $(X,D)$ is not klt. For large $n$, we write
$nL=K+D+A$, where $A$ is ample. If $c$ is largest such that $(X,cD)$ is
log canonical, then we write
\[
nL=K+cD +A^\prime
\]
where $A^\prime=A+(1-c)D$ is still ample. Now, if $cD$ contains a
\emph{divisor} $S$ with multiplicity $1$, then we may assume by
induction on the dimension---and working with the pair
$(S,cD_{|S})$---that all large multiples $nL_{|S}$ are base point
free; standard vanishing theorems then imply that $H^0(X,nL)\rightarrow
H^0(S,nL_{|S})$ is surjective, and we get the result. In general, the
pair $(X,cD)$ has some non-klt centres but they may all have higher
codimension. The traditional approach is to blow up until one of the
centres becomes a divisor (in fact, one blows up indiscriminately to a log
resolution of everything in sight). Instead, in these notes, we use
multiplier ideals, Nadel Vanishing, and the subadjunction theorem of
Kawamata, to work \emph{directly on $X$} and lift sections from a
minimal non-klt centre $W\subset X$ of the pair $(X,cD)$.

Our second aim is to write up a complete, detailed proof of existence of
flips in dimension $n$ assuming the minimal model program with
scaling in dimension $n-1$. Our proof is closer in spirit to some of
Shokurov's original arguments in \cite{Sho03} than to the
treatment of \cite{HM05} and, especially, \cite{HM08}. Our aim
is to present as \textsl{robust} a proof as possible. In
short, we proceed as follows. The starting point is a pl flipping
contraction $f\colon (X,S+B) \rightarrow Z$:
\begin{itemize}
\item $X$ is $n$-dimensional and the pair $(X,S+B)$ has plt singularities;
\item $f$ is a flipping contraction for $K+S+B$;
\item $S$ is non-empty and $f$-negative.
\end{itemize}
The aim is to construct the flip of $f$ by showing that the canonical
algebra
\[
R=R(X,K+S+B)=\bigoplus_{n=0}^\infty H^0\bigl(X,n(K+S+B)\bigr)
\]
is finitely generated. Shokurov's great insight was to suggest that we
can do this by showing that the \emph{restricted algebra}
\[
R_S = \text{Image} \bigl(R \rightarrow k(S)[T]\bigr)
\]
is finitely generated. ($R$ is an algebra of rational functions and we
get $R_S$ by restricting those rational functions to $S$.) We develop
Shokurov's language of b-divisors to make sense of the statement that
$R_S$ is a b-divisorial algebra. We then
show that $R_S$ enjoys two key properties:
\begin{enumerate}
\item $R_S$ is an \emph{adjoint algebra}---this notion comes from a
  further major insight of Hacon and M\textsuperscript{c}Kernan and
  the proof of this fact uses their important lifting lemma;
\item $R_S$ is \emph{saturated} in the sense of Shokurov.
\end{enumerate}
Using these two properties, and the minimal model program with scaling
in dimension $n-1$, we can then show that $R_S$ is finitely
generated. Our argument is robust in the sense that the two
key properties---adjointness and saturation---are treated as separate
issues and proved independently of each other.

This note is a cleaned-up version of AC's lectures at the Summer
School in Grenoble \textsl{Geometry of complex projective varieties and the
  minimal model program} 18~June--06~July 2007. We thank the organisers
of that very successful event for providing an ideal setting and
a comprehensively positive atmosphere.

It is our pleasure to acknowledge the influence of Robert Lazarsfeld
on the point of view endorsed in these notes. It was he who
re-iterated Kawamata's suggestion that a more transparent proof of the
fundamental theorems of Mori theory is possible based on multiplier
ideals and the subadjunction theorem. We especially thank Alex
K\"uronya for his careful reading of earlier versions of these notes
and his useful comments. We also thank Christopher Hacon and S\'andor
Kov\'acs.

We tried hard to chase down and remove mistakes from the text; please
accept our apologies for those that must inevitably be still around.

\section{Basic definitions and results}
\label{sec:preliminaries}

\begin{nt} All varieties in this paper are proper, irreducible and
normal over $\C$. We mostly work with projective varieties.
We write $\sim$ for linear equivalence of Weil
divisors and $\equiv$ for numerical equivalence of Cartier
divisors. On a variety $X$, $\WDiv(X)$ denotes the group of Weil divisors,
$\Div(X)$ the group of Cartier divisors and $\Pic(X)=\Div(X)/\sim$.
Subscripts denote either the ring in which the coefficients of
divisors are taken or that the equivalence is relative to a specified
morphism. An ample $\Q$-divisor $A$ is {\em general\/} if $kA$ is a general
member of the linear system $|kA|$ for some $k\gg0$.
\end{nt}

\begin{dfn}
\mbox{}
\begin{enumerate}
\item[1.] A {\em log pair\/} $(X,\Delta)$ consists of a variety $X$
and a divisor $\Delta\in\WDiv(X)_\R$ such that $K_X +
\Delta$ is $\R$-Cartier.
\item[2.] A pair $(X,\Delta)$ is {\em log smooth\/} if $X$ is
  nonsingular and $\Supp\Delta$ has simple normal crossings.
\item[3.] A {\em model\/} over $X$ is a proper birational morphism
$f\colon Y\rightarrow X$.
\item[4.] A {\em log resolution\/} of $(X,\Delta)$ is a model $f\colon
  Y\rightarrow X$ such that $(Y,f^{-1}_*\Delta+\Exc f)$ is log smooth.
\item[5.] A {\em boundary } is a divisor $\Delta= \sum d_i D_i  \in \WDiv(X)_{\R}$ such that $0 \leq d_i \leq 1$ for all $i$.
\end{enumerate}
\end{dfn}
Pairs first arose through the study of open varieties (Iitaka program).
If $U$ is quasi-projective, and if $X$ is a compactification of $U$ such that $\Delta = X \setminus U$
  has simple normal crossings, then the ring
\[
R(X,\Delta)=\bigoplus_{n \in \N} H^0(X, n(K_X+\Delta))
\]
depends only on the open variety $U$ \cite[Chapters 10 and 11]{Ii82}.

A valuation $\nu\colon k(X)\rightarrow\Z$ is {\em geometric\/} if
$\nu=\mult_E$, where $E\subset Y$ is a prime divisor in a model
$Y\rightarrow X$. Divisors $E\subset Y\rightarrow X$ and $E'\subset
Y'\rightarrow X$ define the same geometric valuation if and only if
the induced birational map $Y\dashrightarrow Y'$ is an isomorphism at
the generic points of $E$ and $E'$. The {\em centre\/} of
a geometric valuation $\nu$ on $X$ associated to a divisor $E$ on a
model $f\colon Y\rightarrow X$ is denoted by $c_X\nu=f(E)$. A geometric valuation
$\nu$ is {\em exceptional\/} when $\codim_X(c_X\nu)\geq2$. We often identify a geometric valuation $\nu$ and the
corresponding divisor $E$.

\begin{dfn}
An integral {\em b-divisor\/} $\D$ on $X$ is an element of the group
$$\bDiv(X)=\underleftarrow{\lim}\WDiv(Y),$$
where the limit is taken over all models $f\colon Y\rightarrow X$
with the induced homomorphisms $f_*\colon\WDiv(Y)\rightarrow\WDiv(X)$.
Thus $\D$ is a collection of divisors $\D_Y\in\WDiv(Y)$ compatible with
push-forwards. Each $\D_Y$ is the {\em trace\/} of $\D$ on $Y$.
\end{dfn}
For every model $f\colon Y\rightarrow X$ the induced map $f_*\colon
\bDiv(Y)\rightarrow\bDiv(X)$ is an isomorphism,
so b-divisors on $X$ can be identified with b-divisors on any model over $X$.
For every open subset $U\subset X$ we naturally define the restriction
$\D_{|U}$ of a b-divisor $\D$ on $X$.
\begin{dfn} The \emph{b-divisor of a nonzero rational function $\varphi$} is
$$\bdiv_X\varphi=\sum\nu_E(\varphi)E,$$
where $E$ runs through geometric valuations with centre on $X$.
Two b-divisors are {\em linearly equivalent} if they
differ by the b-divisor of a nonzero rational function.
\noindent \par The associated
{\em b-divisorial sheaf} $\OO_X(\D)$ is defined by
$$\Gamma(U,\OO_X(\D))=\{\varphi\in k(X):(\bdiv_X\varphi+\D)_{|U}\geq0\}.$$
\end{dfn}
\begin{dfn}
The {\em proper transform\/} b-divisor $\widehat{D}$ of an
$\R$-divisor $D$ has trace $\widehat{D}_Y=f_*^{-1}D$ on every model
$f\colon Y\rightarrow X$.
\noindent \par
The {\em Cartier closure\/} of an $\R$-Cartier divisor $D$ on $X$ is
the b-divisor $\overline{D}$ with trace $\overline{D}_Y=f^*D$ on every model
$f\colon Y\rightarrow X$.
\noindent \par
A b-divisor $\D$ {\em descends\/} to a model $Y\rightarrow X$ if $\D=\overline{\D_Y}$; we then say that $\D$ is a {\em Cartier b-divisor\/}.

A b-divisor $\D$ on $X$ is {\em good} on a model $Y \rightarrow X$ if $\D \geq \overline{\D_Y}$; in particular the sheaf $\OO(\D)$ is coherent.

\noindent \par
A b-divisor $\M$ on $X$ is {\em mobile\/} if it descends to a model
$Y\rightarrow X$ where $\M_Y$ is basepoint free.
\end{dfn}
Note that if a mobile b-divisor $\M$ descends to a model
$W\rightarrow X$, then $\M_W$ is free and $H^0(X,\M)=H^0(W,\M_W)$.
\begin{lem}\label{mobile}
Let $\NN$ be a good b-divisor on $X$. There is a mobile b-divisor $\M$ such that $\M_Y=\Mob\NN_Y$
for every model $f\colon Y\rightarrow X$ and $H^0(X,\NN)=H^0(Y,\NN_Y)$.
\end{lem}
\begin{proof}
Since $\NN_Y=f^*\NN_X+E$ for some effective and exceptional divisor $E$, we have $f_*\Mob\NN_Y=f_*\Mob f^*\NN_X=\Mob\NN_X$.
\end{proof}
\vspace{2mm}
\paragraph{\bf Cartier restriction.}
Let $\D$ be a Cartier b-divisor on $X$ and let $S$ be a normal
prime divisor in $X$ such that $S\not\subset\Supp\D_X$.
Let $f\colon Y\rightarrow X$ be a log resolution of $(X,S)$ such that
$\D$ descends to $Y$. Define the {\em restriction\/} of $\D$
to $S$ as
$$\D_{|S}:=\overline{\D_{Y|\widehat S_Y}}.$$
This is a b-divisor on $S$ via $(f_{|\widehat S_Y})_*$ that does not
depend on the choice of log resolution.
By definition, $\D_{|S}$ is a Cartier b-divisor that satisfies
$(\D_1+\D_2)_{|S}={\D_1}_{|S}+{\D_2}_{|S}$, and
${\D_1}_{|S}\geq{\D_2}_{|S}$ if $\D_1\geq\D_2$.
\begin{dfn}
\mbox{}
\begin{enumerate}
\item[1.] The {\em canonical\/} b-divisor $\K_X$ on $X$ has trace
  $(\K_X)_Y=K_Y$ on every model $Y\rightarrow X$.
\item[2.] The {\em discrepancy\/} $\A(X,\Delta)$ of the pair $(X,\Delta)$ is
$$\A(X,\Delta)=\K_X-\overline{K_X+\Delta}.$$
\item[3.] For a geometric valuation $E$ on $X$, the {\em discrepancy of
    $E$\/} with respect to $(X,\Delta)$ is
$$a(E,X,\Delta)=\mult_E\A(X,\Delta).$$
\end{enumerate}
\end{dfn}
\begin{dfn}  \label{dfn:51}
Let $(X,\Delta)$ be a pair and let $D$ be an $\R$-Cartier divisor on $X$.
The \emph{multiplier ideal sheaf} associated to $(X,\Delta)$ and $D$ is
$$\mathcal{J}((X,\Delta); D)=\OO_X(\lceil\A(X,\Delta+D)\rceil).$$
\end{dfn}
\begin{rem}\label{rem:4}
We obviously have $\mathcal{J}((X,\Delta); D)=\mathcal{J}((X,\Delta+D);0)$. When no confusion is likely
we write $\mathcal{J}(\Delta)=\mathcal{J}((X,\Delta);0)$.
\end{rem}
\begin{thm}[Local vanishing]
Let $f \colon Y\rightarrow X$ be a log resolution of a pair $(X,\Delta)$. Then
$$R^i f_*\OO_Y(\lceil\A(X,\Delta)_Y\rceil)=0$$
for all $i>0$.
\end{thm}
\begin{proof}
If $X$ is projective, consider an ample divisor $A$
such that for $m\gg0$, the divisors $mA-\Delta$ are ample, the sheaves $R^i f_*\OO_Y(K_Y-\lfloor f^{\ast}\Delta\rfloor)\otimes\OO_X(mA)$
are globally generated (when non-zero ), and
$$H^j(X,R^i f_*\OO_Y(K_Y-\lfloor f^{\ast}\Delta\rfloor)\otimes\OO_X(mA))=(0)$$
for all $j>0$ and $i\geq0$. The Leray spectral sequence and Kawamata-Viehweg vanishing yield
\begin{multline*}
H^0(X,R^i f_*\OO_Y(K_Y-\lfloor f^{\ast}\Delta\rfloor)\otimes\OO_X(mA))\\
=H^i(Y,\OO_Y(K_Y-\lfloor f^{\ast}\Delta\rfloor+f^*(mA)))=(0).
\end{multline*}
In particular, the sheaves $R^i f_*\OO_Y(K_Y-\lfloor f^{\ast}\Delta\rfloor)\otimes\OO_X(mA)$ are zero because they are globally generated, and thus
$$R^i f_{\ast}\OO_Y(\lceil\A(X,\Delta)_Y\rceil)=R^i f_{\ast}\OO_Y(K_Y-\lfloor f^{\ast} \Delta
\rfloor)\otimes \OO_X(-K_X)=0$$
for $i>0$.

For the non projective case, see
\cite[Theorem 9.4.1]{Laz04}.
\end{proof}
\begin{thm}[Nadel vanishing]
\label{thm:4}
Let $L$ be a Cartier divisor on $X$ such that $L-(K_X+\Delta)$ is nef and
big. Then
$$H^i\big(X, \OO_X(L) \otimes\mathcal{J}(\Delta)\big)=(0)$$
for all $i>0$.
\end{thm}
\begin{proof}
Let $f \colon Y \rightarrow X$ be a log resolution of $(X,
\Delta)$. The Local vanishing theorem states that
$$R^i f_*\OO_Y(\lceil\A(X,\Delta)_Y\rceil+f^{\ast}L)=0$$
for $i>0$; the Leray spectral sequence and Kawamata-Viehweg vanishing imply
\[
H^i(X,\OO_X(L) \otimes
\mathcal{J}(\Delta)\big))= H^i(Y, \lceil\A(X,\Delta)_Y\rceil+ f^{\ast} L)=(0)
\]
for $i>0$.
\end{proof}
The following result is a characterisation of \emph{big}
divisors.
\begin{dfnlm}[Kodaira's Lemma]
\label{lem:61} Let $D$ be a Cartier divisor on a projective variety $X$
of dimension $n$. The divisor $D$ is \emph{big} if one of the
following equivalent conditions holds:
\begin{enumerate}
\item[1.] There exists $C >0$ such
  that $h^0(X, kD) > Ck^n$ for $k\gg1$.
\item[2.] For any ample divisor $A$ on $X$, there exists $m \in \N$ and
  an effective divisor $E$ on $X$ such that $mD \sim A+E$.
\end{enumerate}
If $D$ is nef, $D$ is big if and only if one of the
following conditions holds:
\begin{enumerate}
\item[1.] $D^n >0$.
\item[2.] There is an effective divisor $E$ and ample $\Q$-divisors
  $A_k$ such that $D \equiv A_k+(1/k)E$ for $k\gg1$.
\end{enumerate}
\end{dfnlm}
We recall some definitions of singularities of pairs.
\begin{dfn}\label{dfn:2}
\mbox{}
\begin{enumerate}
\item[1.] A log pair $(X, \Delta)$ is \emph{Kawamata log terminal (klt)} if
  $\lceil\A(X,\Delta)\rceil\geq0$, or equivalently if
  $\mcal{J}(\Delta)=\OO_X$.
\item[2.] A log pair $(X, \Delta)$ is \emph{purely log terminal (plt)} if
  $a(E,X,\Delta) > -1$ for every exceptional geometric valuation $E$
  on $X$; in particular $\lceil \Delta \rceil$ is reduced and it can be proved that the connected components of $\lceil \Delta \rceil$ are normal.
\item[3.] A log pair $(X, \Delta)$ is \emph{log canonical (lc)} if
  $a(E,X,\Delta) \geq -1$ for every geometric valuation $E$ on $X$.
\item[4.] A log pair $(X, \Delta)$ is \emph{divisorially log terminal (dlt)} if
there is a log resolution $Y\rightarrow X$ such that
$\lceil\A(X,\Delta)_Y+\widehat\Delta_Y\rceil\geq0$.
\end{enumerate}
\end{dfn}
\begin{dfnlm}\label{dfn:5}
Let $(X,\Delta)$ be a lc pair.
  \begin{enumerate}
  \item[1.] The \emph{non-klt locus} of $(X,\Delta)$ is
\[ \nklt(X,\Delta)=\Supp \big(\OO_X/\mcal{J}(\Delta)\big).\]
\item[2.] A \emph{non-klt centre} is the centre $W$ of a geometric
  valuation such that
$W\subset\nklt(X,\Delta)$; therefore
$\mathcal{J}(\Delta)\subset \mathcal{I}_W$.
\item[3.] A non-klt centre $W$ is \emph{isolated} if for any geometric
  valuation $E$ on $X$ such that
  $a(E,X, \Delta)=-1$, $c_XE=W$. A non-klt centre $W$ is
  \emph{exceptional} if it is isolated and if there is a unique geometric
  valuation $E$ on $X$ with $a(E,X, \Delta)=-1$.
\item[4.]If $W$ is an isolated
  non-klt centre, $\mathcal{J}(\Delta)=\mathcal{I}_W$.
\end{enumerate}
\end{dfnlm}
\begin{proof}
Let $W$ be an isolated non-klt centre and let $f\colon Y\rightarrow X$ be a log resolution.
Then $\lceil \A(X,\Delta)_Y \rceil=-E+A$, where $E$ and $A$ are effective, each component of $E$ maps onto $W$ and $A$ is $f$-exceptional.
Thus $f_*\OO_Y(A)=\OO_X$
and since $R^1f_*\OO_Y(-E+A)=0$ by Local vanishing, applying $f_*$ to the sequence
$$0\rightarrow\OO_Y(-E+A)\rightarrow\OO_Y(A)\rightarrow\OO_E(A_{|E})\rightarrow0$$
yields
$$0\rightarrow \mathcal{J}(\Delta)\rightarrow\OO_X\rightarrow f_*\OO_E(A_{|E})\rightarrow0.$$
Therefore $f_*\OO_E(A_{|E})$ is a quotient of $\OO_X$ which is an $\OO_W$-sheaf since $f(E)=W$, so $f_*\OO_E(A_{|E})=\OO_W$ and
$\mathcal{J}(\Delta)=\mcal I_W$.
\end{proof}
\begin{rem}
\label{rem:1}
Let $f \colon Y \rightarrow X$ be a log resolution of $(X, \Delta)$ and for $1 \leq i \leq m$,
 let $E_{i}\subset Y$ be the geometric valuations on
$X$ with $a(E_i, X, \Delta)=-1$. The
non-klt centres of $(X, \Delta)$ are precisely the images $f(\cap_{i \in N}E_i)$ for
$N\subset\{1, \ldots, m\}$. In particular, for any $x \in
\nklt(X, \Delta)$, there is a well defined \emph{minimal non-klt centre}
through $x$.
\end{rem}
\begin{lem}[Tie breaking, {\cite[Proposition 8.7.1]{Kol07}}]
\label{lem:11}
Let $(X, \Delta)$ be a klt pair and $D$ a $\Q$-Cartier divisor on $X$ such
that $(X,\Delta+D)$ is lc. Let $W$ be a
minimal non-klt centre of $(X, \Delta+D)$ and let $A$ be an ample divisor. Then there are
arbitrarily small $\varepsilon, \eta>0$
and a divisor $D' \sim_{\Q} A$ such that $W$ is an exceptional non-klt centre of
$(X,\Delta+(1-\varepsilon)D+ \eta D')$.
\end{lem}
\begin{proof}
We show how to make the non-klt centre $W$ isolated.
 Let $f \colon Y \rightarrow X$ be a log resolution of the pair $(X,
 \Delta+D)$. We write
 \[
   K_Y=f^{\ast}(K_X+\Delta+D)+ \sum a_i E_i  \quad\mbox{and}\quad
 f^{\ast}D = \sum b_i E_i,
 \]
so that
\[
K_Y=f^{\ast}(K_X+\Delta + (1- \varepsilon)D)+ \sum(a_i+\varepsilon b_i)E_i.
\]
Since $(X, \Delta+D)$ is lc and $(X, \Delta)$ is klt, $a_i \geq -1$, $a_i+b_i >-1$ and hence $a_i + \varepsilon b_i>-1$ for all
$\varepsilon>0$.

Let $Z$ be a Cartier divisor such that $W$ is the only non-klt centre contained in $\Supp Z$ and let $D_0$ be a general member
of the linear system $\vert pA-Z\vert$ for a large $p$.
Consider a $\Q$-divisor $D' = \frac1p(Z+D_0)\sim_\Q A$; in particular $W$ is the only non-klt centre contained in $\Supp D'$.
Write $f^{\ast}D'= \sum d_i E_i$, where we may assume that $f$ is a log resolution of $(X, \Delta+D+D')$. By construction,
if $a_i=-1$ then $d_i>0$ precisely when $f(E_i)=W$.

The pair $(X,
\Delta+(1-\varepsilon) D+ \eta D')$ is lc if and only if
\[
a_i+\varepsilon b_i-\eta d_i \geq -1
\]
for all $i$. If $a_i>-1$, this holds for all sufficiently small $\varepsilon, \eta>0$.
If $a_i=-1$ and $f(E_i) \neq W$, then $a_i+\varepsilon
b_i- \eta d_i>-1$ for all $\varepsilon >0$. Fix $\varepsilon$ and define
\[
\eta = \min\{\varepsilon b_i/d_i: a_i=-1, d_i>0 \}.
\]
By construction, the pair $(X,
\Delta+(1-\varepsilon)D+ \eta D')$ is log canonical and $W$ is an isolated
non-klt centre.

The non-klt centre $W$ can further be made exceptional as in
  \cite[Proposition 8.7.1]{Kol07}.
\end{proof}

\begin{thm}[Kawamata's Subadjunction, {\cite[Theorem 8.6.1]{Kol07}}]
 \label{thm:5}
Let $(X,\Delta)$ be an lc pair. Assume that $W$ is an exceptional lc
centre of $(X,\Delta)$ and let $A$ be an ample $\R$-divisor. Then $W$
is normal and for every $\varepsilon >0$,
\[
(K_X+\Delta + \varepsilon A)_{\vert W}\sim_\R K_W+\Delta_W
\]
for some divisor $\Delta_W$ on $W$ such that the pair $(W,\Delta_W)$ is klt.
\end{thm}
\begin{rem}
\label{rem:5}
The choice of the divisor $\Delta_W$ in its linear equivalence
class is not canonical.
\end{rem}

\begin{thm}
Let $(X,S+B)$ be a plt pair, where $S$ is irreducible. Then
\[
(K_X+S+B)_{\vert S}\sim_\R K_S+B_S
\]
for some divisor $B_S$ on $S$ such that the pair $(S,B_S)$ is klt.
\end{thm}

\begin{rem}
Our notation is not the same as the one used in \cite{KM98}, where $B_S$ is called the {\em different\/} and denoted $\Diff(B)$.
\end{rem}

\begin{dfn}
  \label{dfn:6}
Let $(X,\Delta)$ be a pair and let $x$ be a point in $X$. The \emph{log
  canonical threshold} of $(X,\Delta)$ at $x$ is
\[
\lct(\Delta;x)= \sup\{ c \in \R : (X,c\Delta) \mbox{ is lc
  in the neighbourhood of } x\}.
\]
\end{dfn}
\section{Non-Vanishing, Basepoint Free and Rationality theorems}

We present the Minimal Model Program (MMP) in a relative setting,
that is, given a birational projective morphism $f \colon X
\rightarrow Z$ we consider
the MMP over $Z$. The motivation for this is clear
when, for example, $X$ has a fibration structure over $Z$.
\begin{thm}[Cone Theorem]
\label{thm:1}
Let $(X,\Delta)$ be a dlt pair over $Z$. Then there are $(K_X+\Delta)$-negative rational curves $C_i\subset X$ such that
$$\overline{NE}(X/Z)=\overline{NE}(X/Z)_{K_X+\Delta\geq0}+\sum R_i,$$
where the $(K_X+\Delta)$-negative \emph{extremal rays} $R_i$ are spanned by the classes of $C_i$ and are locally discrete in
$\overline{NE}(X/Z)_{K_X+\Delta<0}$.
\end{thm}
\begin{thm}[Contraction Theorem]
\label{thm:2}
Let $(X,\Delta)$ be a dlt pair over $Z$. Let $R \subset \NE (X/Z)$ be a $(K_X+\Delta)$-negative extremal ray.
There is a morphism $\varphi\colon X \rightarrow X'$ that is characterised by:
\begin{enumerate}
\item[1.] $\varphi_{\ast}\mathcal{O}_X= \mathcal{O}_{X'}$,
\item[2.] an effective curve is
  contracted by $\varphi$ if and only if its class belongs to $R$.
\end{enumerate}
The morphism $\varphi$ is the contraction of $R$.
\end{thm}
\begin{rem}
These two theorems hold for lc pairs
\cite{Amb03}.
\end{rem}
\begin{thm}[Basepoint Free Theorem]
\label{thm:3}
Let $(X,\Delta)$ be a klt pair and let $L$ be a nef Cartier divisor on $X$. Assume
that there exists  $p>0$ such that $pL-(K_X+\Delta)$ is nef and big.
Then the linear system $\vert nL \vert$ is basepoint free for all $n\gg0$.
\end{thm}
\begin{rem}
\label{rem:bpf}
The Zariski counterexample \cite{Zar62} shows that the Basepoint Free theorem does not hold for dlt pairs when $pL-(K_X+\Delta)$ is assumed to be nef and big. It does hold however when $pL-(K_X+\Delta)$ is ample.
\end{rem}
\begin{cor}
\label{cor:1}
Let $\varphi\colon X \rightarrow Y$ be the contraction of an extremal ray $R$.
The sequence
\[
0 \longrightarrow \Pic(Y) \stackrel{\varphi^{\ast}}\longrightarrow  \Pic(X) \longrightarrow
  \Z
\]
is exact, where the last map in the sequence is multiplication by a fixed curve $C$ whose class belongs to $R$.
If $L \cdot C=0$, then $\vert nL \vert$ is basepoint free.
\end{cor}
\begin{dfnlm}
\label{lem:1}
Let $(X, \Delta)$ be a $\Q$-factorial dlt pair, and let $\varphi \colon X \rightarrow Y$ be the
contraction of a $(K_X+\Delta)$-negative extremal ray $R$.
\begin{enumerate}
\item[1.] If $\dim Y < \dim X$, then $\varphi$ is a {\em Mori fibration\/};
\item[2.] If $\dim Y = \dim X$ and $\Exc(\varphi)$ is an
  irreducible divisor, then $\varphi$ is a {\em divisorial contraction\/};
\item[3.] If $\dim Y = \dim X$ and $\codim_X\Exc\varphi\geq2$, i.e.\ $\varphi$ is a {\em small\/} map, then
  $\varphi$ is a {\em flipping contraction\/}.
\end{enumerate}
\end{dfnlm}
\begin{proof}
 The only thing there is to prove is that in the second case, if $E
 \subset \Exc(\varphi)$ is a prime divisor, then $E= \Exc(\varphi)$.
 Let $C$ be any curve that is contracted by $f$. As the class of $C$ belongs to $R$, $E \cdot C<0$ by Lemma~\ref{lem:2},
 and $C \subset E$.
\end{proof}
\begin{rem}\label{rem:3}
If $\varphi \colon X \rightarrow Y$ is a small extremal contraction,
$Y$ is not $\Q$-factorial. In higher
dimensions, a surgery operation in codimension $2$ is necessary to
proceed with the MMP. This operation is the {\em flip\/} of $\varphi$; it
is defined in Section \ref{sec:lt}.
\end{rem}
\begin{lem}[Negativity Lemma, {\cite[Lemma 3.38]{KM98}}]
\label{lem:2}
Let $f \colon Y \rightarrow X$ be a birational morphism and let $E=
\sum a_i E_i$
be an $f$-exceptional divisor on $Y$. Assume that
\[E \equiv _f H+D,\]
where $H$ is $f$-nef and $D$ is an effective divisor that has no
common components with $E$. Then all $a_i \leq 0$.
If $x \in X$ is a point such that $D$ or $H$ is not numerically
trivial on $f^{-1}(x)$, and if $E_i$ is a divisor such that $f(E_i)=x$, then
$a_i<0$.
\end{lem}

The following result is the first step in proving Theorem \ref{thm:3}.

\begin{thm}[Non-Vanishing Theorem]
\label{thm:21}
  Let $(X,\Delta)$ be a klt pair and let $L$ be a nef Cartier divisor
  on $X$.
Assume that there exists $p>0$ such that
$pL-(K_X+\Delta)$ is nef and big. Then the linear system $\vert nL \vert$
is non-empty for all $n\gg0$.
\end{thm}
\begin{proof}
The proof is by induction on $d=\dim X$.

We may assume that $pL-(K_X+\Delta)$ is ample and that $\Delta$ is a $\Q$-divisor. Indeed,
Definition-Lemma~\ref{lem:61} shows that there is an effective divisor $E$ such
that for all $0<\varepsilon\ll1$, $pL-(K_X+\Delta)- \varepsilon E$ is
ample, and $(X, \Delta+\varepsilon E)$ is klt since $(X, \Delta)$ is.

If $L$ is numerically trivial, by
Kawamata-Viehweg vanishing
$$h^0(X,nL)=\chi(\mathcal{O}_X(nL))=\chi(\mathcal{O}_X)=h^0(\mathcal{O}_X)\neq 0$$
and the linear series
$\vert nL \vert$ is non-empty for all $n\geq0$.

Assume that $L$ is not numerically trivial. First, we show that for
$q\gg1$, \[qL \sim K_X+\Delta+A, \]
where $A$ is an ample divisor and the pair $(X,\Delta+A)$ is not klt.
Let $C$ be a curve such that $L\cdot C\neq0$.
Since $pL-(K_X+\Delta)$ is ample, there is a positive integer $m$ such that $(m(pL-(K_X+\Delta)))^{d-1}$ is represented
by $C$ plus an effective $1$-cycle.
Thus $L\cdot(pL-(K_X+\Delta))^{d-1}> 0$, and therefore the intersection number
\begin{align*}
 (qL-(K_X+\Delta))^d&= ((q-p)L+pL-(K_X+\Delta))^d \\
& \geq (q-p)L\cdot(pL-(K_X+\Delta))^{d-1}
\end{align*}
tends to infinity with $q$. By Riemann-Roch and Kawamata-Viehweg vanishing, for $q\gg0$,
\begin{equation*}
  h^0(X,n(qL-(K_X+\Delta)))=\frac{n^d}{d!}(qL-(K_X+\Delta))^d+ O(d-1).
\end{equation*}
If $x \in X$ is a nonsingular point, a parameter count shows that there is
a section $D_{q,n} \in \vert n(qL-(K_X+\Delta))\vert$ with $\mult_x D_{q,n} \geq
(d+1)n$ for $q,n$ sufficiently large.

Fix one such pair $q,n$ and denote $A=\frac{1}{n}D_{q,n}$.
By construction, $(X, \Delta +A)$ is not klt at $x$.
Let
\[
c= \min \{t \in \R : (X,\Delta+tA) \mbox{ is not klt}\} \leq 1
\]
and consider a minimal non-klt centre $W$ of $(X,\Delta+cA)$.
By Lemma~\ref{lem:11}, we may assume that $W$ is an exceptional non-klt
centre of $(X, \Delta+ cA)$.

For $0<\varepsilon\ll1$,
\[qL\sim_{\Q} K_X+\Delta+cA+ \varepsilon A +(1-c-\varepsilon)A,\]
and Theorem~\ref{thm:5} shows that
\begin{equation}
  \label{eq:7}
 qL_{\vert W}\sim_\Q K_W+\Delta_W+ (1-c-\varepsilon)A_{\vert W}
\end{equation}
where $(W,\Delta_W)$ is a klt pair.
As $qL_{\vert W}$ is nef and $
qL_{\vert W}-( K_W+\Delta_W)$ is ample, by induction $\vert qL_{\vert
  W}\vert \neq \emptyset$.
The non-klt centre $W$ is exceptional, therefore by Definition-Lemma~\ref{dfn:5},
$\mathcal{J}((X,\Delta);cA)=\mathcal{I}_W$.
By
Nadel vanishing, \[H^i(X,\mathcal{I}_W(qL))=(0)\]
for $i>0$. The
long exact sequence in cohomology associated to
\begin{equation*}
   0\rightarrow \mathcal{I}_W(qL) \rightarrow \OO_X(qL)
   \rightarrow \OO_W(qL_{\vert W}) \rightarrow 0 \end{equation*}
yields a surjection $H^0(X,qL) \rightarrow H^0(W,qL_{\vert W})$, so the
   linear series $\vert qL \vert$ is non-empty.
\end{proof}
\begin{rem}
  \label{rem:28}
When working with klt pairs, one can apply Kodaira's Lemma to
replace nef and big divisors with ample ones. However, this no
longer holds in the context of dlt pairs.
\end{rem}
\begin{proof}[Proof of Theorem \ref{thm:3}]
By Theorem~\ref{thm:21}, the linear system $\vert nL\vert$ is
non-empty for $n\gg0$; we now show that $\vert
nL\vert$ is in fact basepoint free for $n\gg0$. As is explained in
Remark~\ref{rem:28}, we may assume that $pL-(K_X+\Delta)$ is ample.

For a positive integer $n$, since $\Bs \vert n^u L \vert \subset \Bs \vert n^v L
\vert$ when $u>v$, the Noetherian condition shows that the sequence
$\Bs|n^uL|$ stabilises. Denote by $B_n$ its limit.
If $B_p=B_q= \emptyset$ for two relatively prime integers $p$ and $q$,
then
$\Bs \vert p^{u_0}L \vert=\Bs \vert q^{v_0}L \vert= \emptyset $.
Since every $n\gg0$ can be written as a non-negative linear combination of $p^{u_0}$ and $q^{v_0}$,
$\vert nL\vert$ is basepoint free.

Assume that there is an integer $m$ such that $B_m \neq
\emptyset$. By taking a multiple of $L$, we may assume that $B_m= \Bs \vert mL \vert= \Bs \vert
m^eL \vert$ for all $e\geq 1$.
Let $D$ be a general section of $\vert mL \vert$.
	
Define
\[
c = \inf\{t \in \R : (X,\Delta+tD) \mbox{ is not klt}\}>0,
\]
so that the pair $(X, \Delta+cD)$ is strictly lc. Note that $c\leq 1$
because $D$ is an integral divisor.

We choose an appropriate minimal non-klt centre of $(X,
\Delta+cD)$ as follows. If there is a minimal non-klt centre $W$ of
codimension at least $2$, then $W \subset \Bs \vert mL \vert$. If all
minimal non-klt centres are of codimension $1$, note that $c=1$
and that every component of $D$ is a minimal non-klt centre. Let $W$
be a component of $D$ that intersects $\Bs \vert mL \vert$.

By Lemma~\ref{lem:11}, we may assume that $W$ is an
exceptional non-klt centre of the pair $(X, \Delta+cD)$ and therefore $\mathcal{J}(X,\Delta+cD)=\mathcal{I}_W$
by Definition-Lemma \ref{dfn:5}.

Fix $q=m^v\geq p+m$. Since $L$ is nef we have
$$qL \sim_{\Q} K_X+\Delta+D+A,$$
where $A$ is an ample $\Q$-Cartier divisor. Then
\[
qL\sim_{\Q}K_X+\Delta+cD+\varepsilon A +A_{\varepsilon},
\]
where $A_{\varepsilon}= (1-c)D+(1-\varepsilon)A$ is ample for
$\varepsilon$ sufficiently small. Theorem \ref{thm:5} then gives
\begin{equation*}
  \label{eq:22}
  qL_{\vert W}\sim_{\Q} K_W+\Delta_W+ {A_{\varepsilon}}_{\vert W}.
\end{equation*}
As in the proof of Theorem \ref{thm:21} the map
 \[H^0(X, qL)\rightarrow H^0(W,qL_{\vert W})\]
is surjective, and therefore $\Bs|qL|\cap W=\Bs|qL_{|W}|$. But by induction on the dimension
$\Bs|qL_{\vert W}|=\emptyset$ and that is a contradiction. 

\end{proof}

\begin{rem}[Effective basepoint free theorem {\cite{Kol93}}] Under the hypotheses of Theorem~\ref{thm:3}, one can show that there is a positive integer
$m$ that depends only on $\dim X$ and $p$ such that $\vert nL \vert$ is basepoint free for $n\geq m$.
\end{rem}

The following lemma will be used in the proof of the Rationality theorem.

\begin{lem}[\cite{KM98}]
    \label{lem:28}
Let $P(x,y)$ be a non-trivial polynomial of degree at most $n$.
Fix a positive integer $a$ and a positive real number
$\varepsilon$. Assume that $P$ vanishes for
all sufficiently large $x, y \in \N$  such that $0<ay-rx <\varepsilon$ for
some $r \in \R$. Then $r$ is rational and if $r= u/v$ with
$u,v \in \N$ and $\hcf(u,v)=1$, $v \leq a(n+1)/\varepsilon$.
  \end{lem}
  \begin{proof}
    Assume that $r$ is irrational. There are infinitely many pairs
    $(p,q)\in \N^2$ such that $0<\vert aq-rp \vert <
    \varepsilon/(n+2)$ and in particular, there is a large integral
    zero $(p_1,q_1)\in \N^2$ of $P$ in that range. Since all
    pairs
$(kp_1,kq_1)$ for $0
    \leq k \leq n+1$ are also zeroes of $P$, $q_1x-p_1y$ divides
    $P$. Repeating the process for $n+1$ distinct integral solutions
    $\{(p_i,q_i)\}_{1\leq i\leq n+1}$ of $P$ in the range $0<\vert aq-rp \vert <
    \varepsilon/(n+2)$ yields a contradiction.

Now write $r=u/v$ with $\hcf(u,v)=1$.
Fix $k\in \N$ and let $(p,q)\in \N^2$ be an integral solution of
$ay-rx = ak/v$. Since $a(q+lu)-r(p+alv)=ak/v$ for all
$l\in \N$, if $ak/v<\varepsilon$, $ay-rx-ak/v$ divides
$P$. The degree of $P$ is bounded by $n$, hence there are at most $n$
such values of $k$ and $a(n+1)/v\geq \varepsilon$.
  \end{proof}

\begin{thm}[Rationality theorem]
\label{thm:29}
Let $(X, \Delta)$ be a klt pair such
that $K_X+\Delta$ is not nef and let $a>0$ be an integer such that
$a(K_X+\Delta)$ is Cartier. Let $H$ be a nef and big Cartier divisor
 and define:
\[
r=r(H)= \max\{ t\in \R : H+t(K_X+\Delta) \mbox{ is nef\/}\}.
\]
 Then there are positive integers $u$ and $v$ such that $r=u/v$ and
 \[
0<v \leq a (\dim X+1).
\]
\end{thm}
\begin{proof}
We follow closely the proof given in \cite{KM98}.\\[2mm]
\noindent
\emph{Step 1.}
We may assume that the divisor $H$ is basepoint free.
Indeed, Theorem~\ref{thm:3} shows that for $m\gg0$ and
for some $c,d \in \N$, $H'= m(cH+da(K_X+ \Delta))$ is basepoint free.
Moreover, $r(H)$ and $r(H')$ satisfy
\[
r(H)= \frac{r(H')+mda}{mc}.
\]
If the denominator of $r(H')$ is $v'$, the denominator $v$ of $r(H)$
divides $mcv'$. Since $c$ and $m$ can be chosen arbitrarily large, $v$
divides $v'$ and $0<v \leq v' <a(\dim X+1)$.\\[2mm]
\noindent
\emph{Step 2.}
For $(p, q) \in \N^2$, define
\[L(p,q)= \Bs \vert pH+qa(K_X+ \Delta)\vert.\]
Observe that for $p,q$ sufficiently large and such that
$0<aq-rp< \varepsilon$, the sets $L(p,q)$ stabilise.
This follows from the Noetherian condition because if $(p',q')$ is such
that $0<aq'-rp'< \varepsilon$ and $p/q<
p'/q'$, then $L(p',q') \subset L(p,q)$. Denote by $L_0$ their limit and set
\[
I= \{(p,q) \in \Z^2\colon 0<aq-rp<1 \mbox{ and } L(p,q)=L_0\}.
\]
Since $\vert pH +qa(K_X+ \Delta) \vert$ is not nef, and in particular not
basepoint free, the set $L_0$ is not empty.\\[2mm]
\noindent
\emph{Step 3.}
Assume that $r=r(H)$ is not rational.
Let $f \colon Y \rightarrow X$ be a log resolution of $(X, \Delta)$. Consider
the divisors:
$$D_1= f^{\ast}H,\quad
D_2= f^{\ast}(a(K_X+\Delta)),\quad
A= K_Y-f^{\ast}(K_X+\Delta).$$
Since $\lceil A \rceil$ is $f$-exceptional and effective, we have
$$H^0(Y,xD_1+yD_2+\lceil A \rceil) \simeq H^0(Y,xH+ya(K_X+\Delta)).$$
Let $(x,y) \in \N^2$ be such that  $0<ay-rx<1$. Then $xD_1+yD_2+A-K_Y$ is nef
and big, and by Kawamata-Viehweg,
\[P(x,y)=\chi(xD_1+yD_2+\lceil A \rceil)=h^0(xD_1+yD_2+\lceil A \rceil) \]
is a polynomial of degree at most $\dim X$ and is non-trivial by Non-Vanishing.
By Lemma~\ref{lem:28}, there
are arbitrarily large integral points $(p,q)\in \N^2$ such that
$0<aq-rp<1$ and $P(p,q)\neq0$.
Therefore for all $(p,q) \in I$, $\vert pH +qa(K_X+ \Delta)
\vert$ is not empty and hence $L_0 \neq X$.

Let $S$ be a general section of $\vert pH +qa(K_X+ \Delta) \vert$.
The pair $(X,\Delta+S)$ is not klt; define
\[
c = \min\{t \in \R : (X,\Delta+tS) \mbox{ is not klt}\} \leq 1.
\]
 Let $W$ a minimal non-klt
 centre of the pair $(X, \Delta + cS)$. By Lemma~\ref{lem:11}, we may
 assume that $W$ is exceptional.
Nadel vanishing shows that the map
\[H^0(Y,p'H +q'a(K_X+
\Delta))\rightarrow H^0(W, (p'H +q'a(K_X+ \Delta))_{\vert W})\]
is surjective, where the pair $(p',q')$ is chosen so that $aq'-rp'<aq-rp$.
This is a contradiction as in the proof of Theorem \ref{thm:3}.
This proves that $r$ is rational.\\[2mm]
\noindent
\emph{Step 4.}
 Let $r= u/v$, where $u,v \in \N$ and
 $\hcf(u,v)=1$. Assume that $v>a(\dim X+1)$.

 Lemma~\ref{lem:28} with $\varepsilon=1$ shows that there are arbitrarily
 large positive integers $p,q$ with $0<aq-rp<1$ such that $P(p,q) \neq
 0$. The linear system $\vert pH +qa(K_X+ \Delta) \vert$ is not empty
 for all $(p,q) \in I$. We proceed as in Step 3.
\end{proof}
\section{Log terminal and log canonical models}\label{sec:lt}

In this section, we define minimal and canonical models of varieties.
For pairs, we introduce \emph{log
canonical} and \emph{log terminal} models.
\begin{dfn}
\label{dfn:11}
Let $f \colon (X,\Delta) \rightarrow Z$  be a projective morphism.
\begin{enumerate}
\item[1.]
The pair $(X, \Delta)$ is a \emph{log terminal} (lt) model over
$Z$ if it has dlt singularities and if $K_X+\Delta$ is $f$-nef.
\item[2.]
The pair $(X, \Delta)$ is a \emph{log canonical} (lc) model over
$Z$ if it has lc singularities and if $K_X+\Delta$ is $f$-ample.
\end{enumerate}
\end{dfn}
\begin{dfn}
\label{dfn:10}
A \emph{log terminal model} of a dlt pair $f \colon (X, \Delta_X)
\rightarrow Z$ is a commutative diagram
\[\xymatrix{ (X,\Delta_X) \ar@{-->}[rr]^{\varphi} \ar[dr]_f & \quad & (Y,\Delta_Y)
  \ar[dl]^{g}\\
\quad & Z & \quad
}\] where $\varphi$ is birational and:
\begin{enumerate}
\item[$1$.] $g$ is projective,
\item[$2$.] $\varphi^{-1}$ has no exceptional
  divisors, i.e. $\varphi$ is contracting,
\item[$3$.] $\Delta_Y= \varphi_{\ast} \Delta_X$,
\item[$4$.] $K_Y +\Delta_Y$ is $g$-nef,
\item[$5$.] $a(E,X,\Delta_X) < a(E,Y,\Delta_Y)$ for every
  $\varphi$-exceptional divisor $E$.
\end{enumerate}
\end{dfn}
\begin{rem}\mbox{}
  \begin{itemize}
  \item[(a)]  This definition corresponds to that of \emph{weak canonical} models
 in \cite[Definition 3.50]{KM98}. Note however that \cite{KM98}
 only assumes that $g$ is proper.
\item[(b)] Log terminal models can be defined without requiring
  $\varphi$ to be contracting. We add this condition
  for clarity of the exposition.
  \end{itemize}
\end{rem}
\begin{dfn}
A \emph{log canonical model} is a diagram as in
Definition~\ref{dfn:10}, satisfying conditions $(1-3)$ above and:
\begin{enumerate}
\item[$4'$.]$K_Y +\Delta_Y$ is $g$-ample,
\item[$5'$.] $a(E,X,\Delta_X) \leq a(E,Y,\Delta_Y)$ for every
  $\varphi$-exceptional divisor $E$.
\end{enumerate}
A \emph{weak log canonical model} is a diagram as in
Definition~\ref{dfn:10}, satisfying conditions $(1-4)$ and $(5')$.
\end{dfn}
\begin{lem}\label{lem:27}
Assume conditions $(1-4)$ of
Definition~\ref{dfn:10}. If $K_X+\Delta_X$ is $f$-negative, then
condition $(5)$ automatically holds.
\end{lem}
\begin{proof}
Let $W$ be a common resolution of $(X, \Delta_X)$ and $(Y, \Delta_Y)$.
\[\xymatrix{ \quad & W \ar[dr]^q \ar[dl]_p & \quad \\
(X,\Delta_X) \ar@{-->}[rr]^{\varphi} & \quad & (Y,\Delta_{Y}) \\
}\]
Write
\begin{align*}
  K_W &= p^{\ast}(K_X+\Delta_X)+\sum a(E, X, \Delta_X)E\\
 &= q^{\ast}(K_Y+\Delta_Y)+\sum a(E, Y,
\Delta_Y)E,
\end{align*}
where $E$ runs through divisors on $W$. The result follows from the Negativity Lemma since
$q^{\ast}(K_Y+\Delta_Y)-p^*(K_X+\Delta_X)$ is $(f\circ p)$-nef and non-trivial on $\varphi$-exceptional divisors.
\end{proof}
\begin{dfn}
Let $(X,\Delta)$ be an lc pair and let $f\colon X\rightarrow Z$ be a projective morphism. The {\em canonical ring\/} of $(X,\Delta)$ over $Z$ is
$$R(X,K_X+\Delta)=\bigoplus_{n\geq0}f_*\OO_X(\lfloor n(K_X+\Delta)\rfloor).$$
\end{dfn}
\begin{lem}
\label{lem:65}
Let $(X, \Delta_X)$ be a dlt (respectively klt, lc) pair over $Z$.
\begin{enumerate}
\item[1.] A log terminal (respectively log canonical) model $(Y,\Delta_Y)$ of $(X,
  \Delta_X)$ is dlt (respectively klt, lc).
\item[2.] If $(Y,\Delta_Y)$ is a log terminal or log canonical model of $(X, \Delta_X)$,
then $R(Y, K_Y+\Delta_Y)= R(X,K_X+\Delta_X)$.
\item[3.] If $(Y,\Delta_Y)$ is a log canonical model of $(X, \Delta_X)$, then
$$Y \simeq\Proj_Z R(X,K_X+\Delta_X).$$
In particular, there is a unique log canonical model of $(X, \Delta_X)$.
\item[4.] Any two log terminal models of $(X, \Delta_X)$ are isomorphic in codimension $1$.
\end{enumerate}
\end{lem}
\begin{proof}
 These are consequences of the Negativity Lemma. The main difficulty
 is to show that discrepancies increase in flips or divisorial
 contractions of dlt pairs \cite[Corollary 3.44]{KM98}.
\end{proof}

The following definition and lemma resolve the issue raised in Remark \ref{rem:3}.

\begin{dfn}\label{dfn:4}
Let $\varphi \colon (X, \Delta_X) \rightarrow Z$ be a flipping
contraction, i.e.:
\begin{enumerate}
\item[$1.$] $\varphi$ is a small projective morphism,
\item[$2.$] $-(K_X+\Delta_X)$ is $\varphi$-ample,
\item[$3.$] $\rho(X/Z)=1$.
\end{enumerate}
A \emph{flip of
  $\varphi$} is a commutative diagram\[
\xymatrix{ (X,\Delta_X) \ar@{-->}[rr]\ar[dr]_{\varphi} & \quad & (X^+,\Delta_{X^+}) \ar[dl]^{\varphi^+}\\
 \quad & Z & \quad } \]
where:
\begin{enumerate}
\item[$1^+.$]  $\varphi^+$ is small and projective,
\item[$2^+.$] $K_{X^+}+\Delta_{X^+}$ is $\varphi^+$-ample, with $\Delta_{X^+}=((\varphi^+)^{-1}\circ\varphi)_*\Delta_X$,
\item[$3^+.$] $\rho(X^+/Z)=1$.
\end{enumerate}
\end{dfn}
\begin{lem}
\label{lem:3}
Let $(X, \Delta_X)$ be a $\Q$-factorial dlt pair and let $\varphi \colon X
\dashrightarrow X'$ be a divisorial contraction, a Mori fibration or a flip. Then
$X'$ is $\Q$-factorial.
\end{lem}
\begin{proof}
\cite[Propositions 3.36-3.37]{KM98}.
\end{proof}
\begin{con}[Existence of flips]
\label{conj:exist}
Dlt flips exist.
\end{con}
\begin{rem}\label{rem:6}
Let $\varphi\colon(X, \Delta)\rightarrow Z$ a flipping contraction. The flip of $\varphi$ exists if and only if the canonical algebra of $(X,\Delta)$ is
finitely generated \cite[Corollary 6.4]{KM98}. Therefore the problem of existence of flips is local on the base $Z$, and we often assume
that $Z$ is affine.
\end{rem}
\begin{con}[Termination of flips]
 There does not exist an infinite sequence of dlt flips.
\end{con}
\begin{rem}
  \label{rem:40}
  A $\Q$-factorial dlt pair $(X, \Delta)$ is a limit of klt pairs. For all $0<\delta\ll1$ the pair $(X, (1-\delta)\Delta)$ is klt as in the proof of  \cite[Proposition 2.43]{KM98}.
The existence of flips of klt contractions therefore implies that of flips of
$\Q$-factorial dlt contractions.
\end{rem}
\begin{dfn} Let $(X, \Delta)$ be a pair.
\begin{itemize}
\item[1.] $(X, \Delta)$ is a \emph{minimal model\/} if $K_X+\Delta$ is nef.
\item[2.] If $(X, \Delta)$ is not nef and there is a Mori fibration $(X, \Delta) \rightarrow Y$, then $(X, \Delta)$ is a \emph{Mori fibre space\/}.
\end{itemize}
\end{dfn}

\paragraph{\bf Minimal Model Program (MMP)}
Let $(X,\Delta_X)$ be a dlt
pair over $Z$.
If existence and termination of flips hold, there is a sequence
of birational maps
$$X=X_1\dashrightarrow X_2\dashrightarrow \cdots \dashrightarrow X_n=X'$$
over $Z$, where $(X', \Delta_{X'})$ is either a minimal model or a
Mori fibre space.
The sequence is obtained as in Flowchart \ref{fig:mmp}.
\begin{figure}[htb]
\begin{center}
\includegraphics[width=\textwidth]{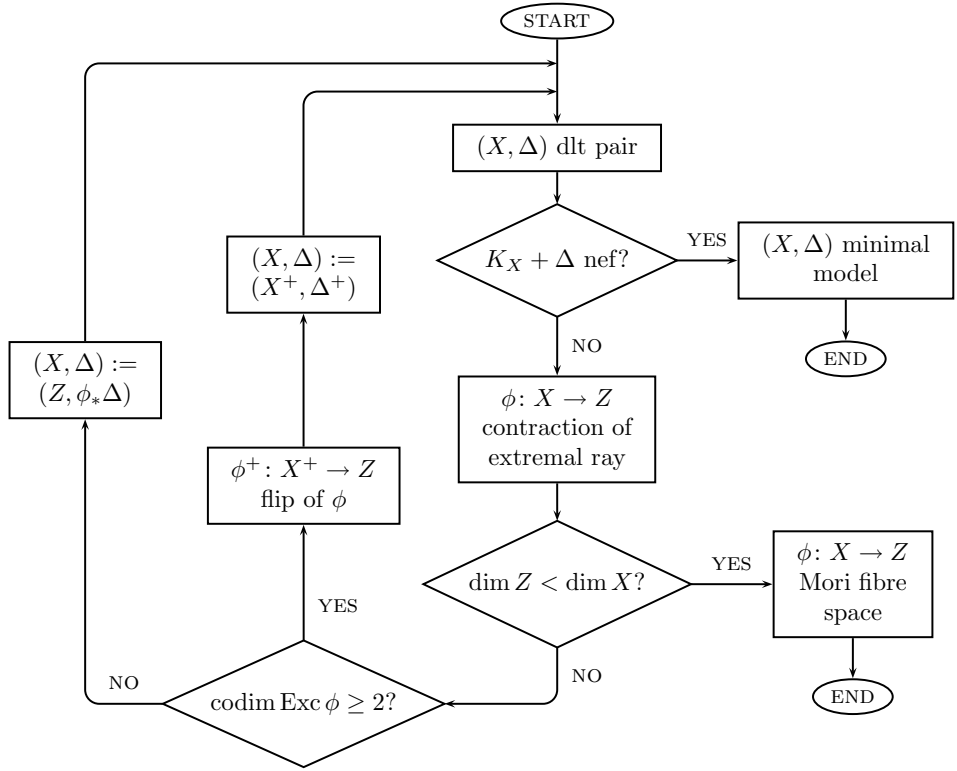}
\end{center}
\caption{Flowchart of the Minimal Model Program}
\label{fig:mmp}
\end{figure}
\begin{rem}
 \label{rem:2}
The MMP can be run for non $\Q$-factorial log
canonical pairs $(X, \Delta_X)$, where $\Delta_X$ is an $\R$-divisor.
The MMP for non
$\Q$-factorial klt pairs is discussed in \cite{Fuj07a}.
\end{rem}
\begin{pro}\label{prop:15}
Let $f\colon (X, \Delta_X) \rightarrow Z$ be a
flipping contraction, where $(X, \Delta_X)$ is a $\Q$-factorial
dlt pair. Assume that $g\colon (Y,\Delta_Y)
\rightarrow Z$ is a log terminal model of $f$. Then:
\begin{enumerate}
\item[1.] $(Y,\Delta_Y)$ is also a log canonical model of $f$.
\item[2.] $(Y,\Delta_Y)$ is the flip of $f$.
\end{enumerate}
\end{pro}
\begin{proof}  The morphism $g$ is small because $\varphi^{-1}$ has no
  exceptional divisors. We prove that $K_Y+\Delta_Y$ is $g$-ample.

Let $A$ be a $g$-ample divisor on $Y$ and denote $A_X=
\varphi_{\ast}^{-1}A$. Since $X$ is
$\Q$-factorial and $f$ is flipping, $A_X$ is $\Q$-Cartier and
\[
\lambda (K_X+\Delta_X)+A_X \sim_\Q f^{\ast}M,
\]
where $M$ is a $\Q$-Cartier divisor on $Z$ and $\lambda \in
\Q$. Therefore
\[
\lambda (K_Y+\Delta_Y)+A\sim_\Q g^{\ast}M,
\]
and $\lambda$ is non-zero because $A$ is $g$-ample.
The divisor $K_Y+\Delta_Y$ is $g$-nef if and only
if it is $g$-ample.
\end{proof}

\paragraph{\bf MMP with scaling}
\begin{lem}[\cite{Bir07, Sho06}]
  \label{lem:29}
Let $(X, \Delta+A)$ be a $\Q$-factorial dlt pair over
$Z$ such that $K_X+\Delta+A$ is nef.
If $K_X+ \Delta$ is not nef, there is a $(K_X+
\Delta)$-negative extremal ray $R$ and a real number $0<\lambda \leq 1$ such that $K_X + \Delta+ \lambda A$ is nef and trivial
on $R$.
\end{lem}
\begin{proof}Define
  $$\lambda=\inf\{-(K_X+ \Delta)\cdot\Sigma/A\cdot\Sigma\}$$ as $\Sigma$
  ranges over curves generating extremal
rays of $(X,\Delta)$; it is enough to show that $\lambda$ is
a minimum. See \cite{Bir07,Sho06} for a proof.
\end{proof}

\begin{dfn}
\label{directedmmp}
Let $(X, \Delta+A)$ be a $\Q$-factorial dlt pair over
$Z$ such that $K_X+\Delta+A$ is nef.
A {\em $(K_X+\Delta)$-MMP with scaling of $A$\/} is a sequence:
\[
(X_1,\lambda_1) \stackrel{\varphi_1}\dashrightarrow \cdots
  \stackrel{\varphi_{i-1}}\dashrightarrow (X_i,\lambda_i)
  \stackrel{\varphi_i} \dashrightarrow
  \cdots,
\]
where for each index $i$, $A_i$ and $\Delta_i$ are the strict transforms
of $A$ and $\Delta$ on $X_i$, and
 \[\lambda_i=\min\{t \in \R: K_{X_i}+\Delta_i+t A_i
\mbox{ is nef}\}.\]
Each $\varphi_i$ is a divisorial contraction or a flip
associated to a $(K_{X_i}+\Delta_i)$-negative extremal ray $R_i\subset\overline{\NE}(X_i/Z)$ which is
$(K_{X_i}+\Delta_i+\lambda_iA_i)$-trivial.
\end{dfn}

The proof of the existence of minimal models for varieties of log general type given in \cite{BCHM} rests on several theorems; we recall two of
them here.

\begin{thm}[Existence of models,  {\cite[Theorem C]{BCHM}}]
\label{eom}
Let $(X,\Delta)$ be a klt pair projective over $Z$, where $\Delta$ is big over $Z$. If there is an effective divisor $D$ such that
$K_X+\Delta \sim_{\R,Z} D$, then $(X, \Delta)$ has a log terminal model over $Z$.
\end{thm}
\begin{thm}[Finiteness of models, {\cite[Theorem E]{BCHM}}]
\label{fom}
Let $(X,\Delta_0)$ be a klt pair projective over $Z$.
Fix a general ample $\Q$-divisor $A$ over $Z$ and a rational finite
dimensional affine subspace $V\subset\WDiv(X)$ that contains $\Delta_0$.
Con\-si\-der the rational polytope
\[ \mcal L_A=\{A+B : B\in V, B \geq 0 \mbox{ and the pair } (X, A+B) \mbox{ is lc\/}\}. \]

Then the set of isomorphism classes of weak log canonical models of $(X, \Delta)$ over $Z$, where $\Delta \in \mcal{L}_{A}$, is finite.
\end{thm}

Now we are ready to state Special termination with scaling.

\begin{thm}[ {\cite[Lemma 5.1]{BCHM}}]
\label{thm:40}
Assume that existence and finiteness of models hold in dimension $n-1$. Let $(X,
\Delta+C)$ be an $n$-dimensional $\Q$-factorial dlt pair such that $K_X+\Delta+C$ is nef. Assume that $\Delta=A+B$, where $A$ is ample, $B\geq0$
and $\lfloor\Delta\rfloor\subset\Supp B$. Let
\[
(X,\Delta)=(X_1,\Delta_1)
\stackrel{\varphi_1}\dashrightarrow (X_2,\Delta_2)
\stackrel{\varphi_2}\dashrightarrow
\cdots
\]
be a sequence of flips for the $(K_X+\Delta)$-MMP with scaling of $C$. Denote by $R_k$ the exceptional set of the corresponding flipping contraction of
$\varphi_k$. For $k$ sufficiently large, $R_k \cap
\lfloor \Delta_k \rfloor= \emptyset$.
\end{thm}
\begin{dfn}
  \label{dfn:40}
  Let $(X, \Delta)$ be a $\Q$-factorial dlt pair and $f\colon (X, \Delta)\rightarrow Z$ a flipping contraction.
We say $f$ is a \emph{pre limiting (pl) flipping contraction} if there is an $f$-negative irreducible component
$S\subset\lfloor\Delta\rfloor$.
\end{dfn}
\begin{thm}
\label{thm:41}
Assume that existence and finiteness of models hold in dimension $n-1$.
If flips of pl contractions exist in dimension $n$, then flips of klt contractions exist in dimension $n$.
\end{thm}
\begin{proof}
The proof follows the ideas in \cite{Sho92, Kol92, Fuj07a}.

Let $f\colon (X,\Delta)\rightarrow Z$ be a klt flipping
contraction.
We assume that $Z$ is affine.
In this proof subscripts denote proper transforms.\\[2mm]
\noindent
\emph{Step 1.}
Let $\pi \colon W \rightarrow Z$ be a resolution of $Z$ and let $F_i$ be
generators of $N^1(W/Z)$. We may assume that no $\pi(F_i)$ contains an
irreducible component of $f(\Delta)$.
Write $Z'=\Sing Z \cup \Sing f(\Delta)\cup f(\Exc f)$. Note that $\codim_Z Z'\geq2$, and let $\mcal I_{Z'}$ be the ideal sheaf
of $Z'$. Let $i\colon Z\rightarrow W$ be a compactification of $Z$ with a very ample line bundle $\mcal L$ on $W$ such that
$\mcal L_{|Z}=\OO_Z$ and $\mcal F=i_*\mcal I_{Z'}\otimes\mcal L$ is globally generated. If $H$ is the restriction of a general global section of
$\mcal F$ to $Z$, then $H$ is reduced and $f^{\ast}H=H_X$ since $f$ is small. Note that $H_X$ and $\Delta$ have no common components.
By construction, if $\rho \colon W' \rightarrow Z$ is any $\Q$-factorial model
of $Z$, the group $N^1(W'/Z)$ is generated by the components
of $H_{W'}$ and by the $\rho$-ex\-cep\-ti\-o\-nal divisors.

Let $h \colon Y \rightarrow X$ be a log resolution of the pair
$(X, \Delta+H_X)$ such that $f \circ h$ is an isomorphism over $Z
\smallsetminus H$. Note that all the
$h$-exceptional divisors are components of $h^{\ast} H_X$.
Denote by $E_i$ the exceptional divisors of $h$ and let $E=\sum E_i$.
The pair $(Y, \Delta_Y + H_Y+ E)$ is $\Q$-factorial and dlt over $Z$. Consider a very ample divisor $A$ on $Y$ which is general
in $|A|$ such that $K_X+\Delta_Y+H_Y+E+A$ is nef.
Therefore replacing $H$ by $H+(f\circ h)_*A$ we may assume that $K_X+\Delta_Y+H_Y+E$ is nef.\\[2mm]
\noindent
\emph{Step 2.}
We now run the $(K_{\overline{Y}}+\Delta_{\overline{Y}}+E_{\overline{Y}})$-MMP
with scaling of $H_{\overline{Y}}$.
We construct a sequence
\[
(\overline Y_1,\Delta_{\overline Y_1} + E_{\overline Y_1})
\stackrel{\chi_1}\dashrightarrow \cdots
\stackrel{\chi_{i-1}}\dashrightarrow (\overline Y_i,\Delta_{\overline
  Y_i} + E_{\overline Y_i})
\stackrel{\chi_i}\dashrightarrow \cdots,
\]
where $\overline Y_1=\overline Y$, and each $\chi_i$ is a
divisorial contraction or a pl flip.
Denote by $h_i \colon Y_i \rightarrow Z$ the induced map.

If $K_{\overline Y_i}+\Delta_{\overline Y_i} + E_{\overline Y_i}$ is
not nef, let $R_i$ be an extremal ray as in Lemma \ref{lem:29}.
Let $D_i$ be any curve whose class is in $R_i$, in particular
$H_{\overline{Y}_i} \cdot D_i>0$.  Since $D_i$ is $h_i$-exceptional,
the definition of
$H$ and $h_i^*H \cdot D_i=0$ imply $D_i \cdot
E_{\overline{Y}_i}<0$. The curve $D_i$ therefore intersects a component of
$\lfloor E_{\overline{Y}_i}+ \Delta_{\overline Y_i}\rfloor= E_{\overline{Y}_i}$
negatively. If the contraction associated to $R_i$ is small, it is pl flipping.
Define $\chi_i\colon \overline Y_i\dashrightarrow \overline Y_{i+1}$
to be the contraction of $R_i$ if it is divisorial, or its flip if it is small.

Theorem \ref{thm:40} shows that after finitely many steps we obtain a pair
$\tilde{g}\colon(\widetilde{Y},
\Delta_{\tilde{Y}}+E_{\tilde{Y}})\rightarrow Z$ such that
$K_{\tilde{Y}}+\Delta_{\tilde{Y}}+E_{\tilde{Y}}$ is $\tilde{g}$-nef.\\[2mm]
\noindent
\emph{Step 3.}
Consider a common resolution
\[\xymatrix{ \quad & W \ar[dr]^q \ar[dl]_p & \quad \\
(X,\Delta) \ar@{-->}[rr] & \quad & (\widetilde
Y,\Delta_{\tilde{Y}}+E_{\tilde{Y}})
\\
}\]
We have
\begin{align*}
K_W&=p^*(K_X+\Delta)-p_*^{-1}\Delta+E_{p,q}+E_p\\
&=q^*(K_{\widetilde Y}+\Delta_{\widetilde Y}+E_{\widetilde
  Y})-q_*^{-1} \Delta_{\widetilde Y}-q_*^{-1}E_{\widetilde
  Y}+\widetilde E_{p,q}+
\widetilde E_q,
\end{align*}
where the divisors $E_{p,q}$ and $\widetilde E_{p,q}$ are exceptional
for $p$ and $q$, $E_p$ is exceptional for $p$ but not for $q$, and
$\widetilde E_q$ is exceptional for $q$ but not for $p$.
Observe that $q_*^{-1}E_{\widetilde Y}$ is exceptional for $p$ but not for $q$.

The divisor $q^*(K_{\widetilde Y}+\Delta_{\widetilde Y}+E_{\widetilde Y})-p^*(K_X+\Delta)$
is $(\tilde g\circ q)$-nef, and $p_*^{-1}\Delta-q_*^{-1} \Delta_{\widetilde Y}$
is effective and exceptional for $q$ but not for $p$. The Negativity
Lemma over $Z$ shows that
\begin{equation}\label{eq:1}
E_{p,q}<\widetilde E_{p,q}\quad\mbox{and}\quad E_p+q_*^{-1}E_{\widetilde Y}\leq0.
\end{equation}
Since $(X,\Delta)$ is klt, we have $\lceil E_p\rceil\geq0$, and as
$q_*^{-1}E_{\widetilde Y}$ is reduced, this implies $E_{\widetilde Y}=0$.
In particular, the map $\widetilde g$ is small and \eqref{eq:1}
implies that $(\widetilde Y,\Delta_{\widetilde Y})$ is klt.
The pair $(\widetilde Y,\Delta_{\widetilde Y})$ is a log terminal
model of $(X, \Delta)$, and
by Proposition~\ref{prop:15} this is the required flip.
\end{proof}
\begin{rem}
\label{Sirred}
 The existence of pl flips can be reduced to the case where the pair $(X, S+B)$ is plt, where $S$ is irreducible as in \cite[Remark 2.2.21]{Cor07}.
\end{rem}
\section{Restricted algebras and adjoint algebras}

The rest of the paper is devoted to the proof of Conjecture
\ref{conj:exist} under certain conditions.
More precisely, we prove that
klt flips exist in dimension $n$ if existence and finiteness of models hold in dimension $n-1$.
As is explained in Remark \ref{rem:6}, the existence of flips is local on the base;
for the rest of the paper we consider varieties that are projective over affine varieties.

We work with graded algebras of rational functions $R=\bigoplus_{n\geq0}R_n\subset k(X)[T]$,
where $R_0$ is a finitely generated $H^0(X,\OO_X)$-algebra.

\begin{dfn}
A {\em truncation\/} of $R$ is an algebra of the form
$R^{(I)}=\bigoplus_{n\geq0}R_{nI}$ for a positive integer $I$.
\end{dfn}

We often use the following simple but important lemma without explicit reference.

\begin{lem}\label{truncation}
A graded algebra $R$ is finitely generated if and only if its
truncation $R^{(I)}$ is finitely generated for any $I$.
\end{lem}
\begin{proof}
\cite[Lemma 2.3.3]{Cor07}.
\end{proof}

\begin{dfn}
A sequence of effective $\R$-b-divisors $\M_\bullet$ is {\em
  superadditive\/} if
$\M_{m+n}\geq\M_m+\M_n$ for all $m,n\geq0$.

A sequence of effective $\R$-b-divisors $\DDelta_\bullet$ is {\em concave\/} if
$$\DDelta_{m+n}\geq\frac{m}{n+m}\DDelta_m+\frac{n}{n+m}\DDelta_n$$
for all $n,m\geq0$. Convex sequences are defined similarly.
\end{dfn}

Observe that if $\M_\bullet$ is a superadditive sequence of b-divisors, then there are homomorphisms
$$H^0(X,\M_m)\otimes H^0(X,\M_n)\rightarrow H^0(X,\M_{m+n})$$
for all $m,n\geq0$. This justifies the following definition.

\begin{dfn}
A {\em b-divisorial algebra\/} on $X$ is the algebra of rational functions
$$R(X,\M_\bullet)=\bigoplus_{m\in\N}H^0(X,\M_m),$$
where $\M_\bullet$ is a superadditive sequence of b-divisors on $X$.
\end{dfn}

\begin{lem}\label{limit}
Let $\M_\bullet$ be a superadditive sequence of mobile b-divisors on $X$.
The b-divisorial algebra $R=R(X,\M_\bullet)$ is finitely generated if and only if
there exists an integer $i$ such that $\M_{ik}=k\M_i$ for all $k\geq0$.
\end{lem}
\begin{proof}
Assume that for some $i$, $\M_{ik}=k\M_i$ for all $k\geq 0$.
Passing to a truncation we may assume that $i=1$.
Let $Y \rightarrow X$ be a model such that $\M_1$
descends to $Y$. Then $R=\bigoplus H^0(Y,i\M_{1Y})$ and the result
follows from Zariski's theorem.

Conversely, assume that $R$ is finitely
generated. Up to truncation, we may assume that $R$ is generated by
$H^0(X,\M_1)$.
For each $j$, take a resolution $Y_j\rightarrow X$ such that
both $\M_1$ and $\M_j$ descend to $Y_j$.
Since the sequence $\M_\bullet$ is concave
we have $H^0(Y_j,j\M_{1Y_j})\subset H^0(Y_j,\M_{jY_j})$,
and the finite generation gives
$H^0(Y_j,\M_{jY_j})=H^0(Y_j,\M_{1Y_j})^j \subset H^0(Y_j,j\M_{1Y_j})$.
As $j\M_{1Y_j}$ and $\M_{jY_j}$ are free,
$H^0(Y_j,j\M_{1Y_j})=H^0(Y_j,\M_{jY_j})$ implies
$j\M_{1Y_j}=\M_{jY_j}$ and thus $j\M_1=\M_j$.
\end{proof}


\begin{stp}\label{setup}
Consider a pl flipping contraction
$$\pi\colon(X,\Delta=S+B)\rightarrow Z,$$
where $S$ is a prime divisor and $\lfloor B\rfloor=0$. Recall that $-(K_X+\Delta)$ and $-S$ are $\pi$-ample, that $X$ is $\Q$-factorial and
$\rho(X/Z)=1$.
\end{stp}

The flip of $f$ exists if and only if the canonical algebra
$$R(X,K_X+\Delta)=\bigoplus_{n\geq0}H^0(X,\lfloor n(K_X+\Delta)\rfloor)$$
is finitely generated.


\begin{rem}
For a Cartier divisor $D$ and a prime Cartier divisor $S$, let $\sigma_S\in H^0(X,S)$
be a section such that $\ddiv\sigma_S=S$. From the exact sequence
$$0\rightarrow H^0(X,\OO_X(D-S))\stackrel{\cdot \sigma_S}{\longrightarrow}H^0(X,\OO_X(D))\stackrel{\rho_{D,S}}{\longrightarrow}
H^0(S,\OO_S(D))$$
we denote $\res_S H^0(X,\OO_X(D))=\im(\rho_{D,S})$.
\end{rem}

\begin{dfn}
The {\em restricted algebra} of $R(X,K_X+\Delta)$ is
$$R_S(X,K_X+\Delta)=\bigoplus_{n\geq0}\res_S H^0(X,\lfloor n(K_X+\Delta)\rfloor).$$
\end{dfn}
\begin{lem}
Under assumptions of Setup \ref{setup}, $R(X,K_X+\Delta)$ is finitely generated if and only if $R_S(X,K_X+\Delta)$ is finitely generated.
\end{lem}
\begin{proof}
We will concentrate on sufficiency, since necessity is obvious.

By \cite[Lemma 3-2-5]{KMM87}, numerical and linear equivalence over $Z$ coincide. Since $\rho(X/Z)=1$, and both $S$ and $K_X+\Delta$ are
$\pi$-negative, there exists a positive rational number $r$ such that $S\sim_{\Q,\pi}r(K_X+\Delta)$. By considering open subvarieties of $Z$
we can assume that $S-r(K_X+\Delta)$ is $\Q$-linearly equivalent to a pullback of a principal divisor.

Let $\varphi$ be a rational function with a zero of order one along
$S$ and set $D=S-\ddiv\varphi$; in particular $S\not\subset\Supp D$. Therefore
$$D\sim_{\Q}r(K_X+\Delta),$$
and the rings $R(X,D)$ and $R(X,K_X+\Delta)$ share a common
truncation. As $\varphi\in H^0(X,D)$, it is enough to show that $\varphi$ generates the kernel
$$K=\bigoplus_{n\geq0}\ker\big(H^0(X,nD)\rightarrow H^0(S,nD_{|S})\big).$$

If $\varphi_n\in H^0(X,nD)\cap K$, then
$\ddiv\varphi_n+nD-S\geq0$. Writing $\varphi_n=\varphi\varphi_n'$ for some
$\varphi_n'\in k(X)$, we have
$$\ddiv\varphi_n'+(n-1)D=\ddiv\varphi_n-\ddiv\varphi+(n-1)D\geq0,$$
and therefore $\varphi_n'\in H^0(X,(n-1)D)$.
\end{proof}

\begin{dfn}\label{adjointdefn}
Let $X$ be a variety that is projective over an affine variety $Z$.
Let $\DDelta_m$ be an effective concave sequence of $\Q$-b-divisors on $X$ such that
each $K_X+\DDelta_{mX}$ is $\Q$-Cartier, and assume that there is a
positive integer $I$ such that $mI\DDelta_m$ is integral for every $m$.

An {\em adjoint algebra} is a b-divisorial algebra
$$R(X,\NN_\bullet)=\bigoplus_{m\geq0}H^0(X,\NN_m)$$
where $\NN_m=mI(\K_X+\DDelta_m)$, such that:
\begin{itemize}
\item $\NN_m$ is good on $X$ for every $m\geq0$, that is $\NN_m \geq \overline{\NN_{mX}}$,
\item $\lim\limits_{m\rightarrow\infty}\DDelta_m=\sup\limits_{m\geq0}\DDelta_m=\DDelta$
  exists as a real b-divisor on $X$.
\end{itemize}
\end{dfn}
\begin{rem}
Any truncation of an adjoint algebra is an adjoint algebra.
\end{rem}
\begin{hmck}[Hacon-M\textsuperscript{c}Kernan]
Let $(Y,T+B)$ be a log smooth plt pair that is projective over an affine variety $Z$, where $T$ is irreducible and $B$ is a $\Q$-divisor with
$\lfloor B\rfloor=0$. Assume that
\begin{enumerate}
\item $B\sim_{\Q}A+C$, where $A$ is ample, $C\geq0$ and $T\not\subset\Supp C$,
\item there are a positive integer $p$ and a divisor $N\in|p(K_Y+T+B)|$ such that $N$ meets properly all the intersections of the components
of $T+B$.
\end{enumerate}
Then for every $m$ such that $m(K_Y+T+B)$ is integral, the map
$$H^0(Y,m(K_Y+T+B))\rightarrow H^0(T,m(K_T+B_{|T}))$$
is surjective.
\end{hmck}
\begin{proof}
See \cite{HM07}.
\end{proof}
\begin{rem}
The assumptions of the Lifting Lemma are birational. The first
assumption is related to $B$ being big; in our setting it is not restrictive.
The second assumption, however, is difficult to arrange.
Example \ref{exa:2} shows that the statement cannot be expected
to hold without both assumptions being satisfied.
\end{rem}
\begin{exa}\label{exa:2}
Let $Y$ be a surface, let $T$ be a nonsingular curve on $Y$ and let
$E\subset Y$ be a $(-1)$-curve meeting $T$ at exactly one point
$P$. Denote $B=dE$,
where $0<d<1$. Let $f\colon Y\rightarrow X$ be the contraction of $E$
and let $S=f(T)$. Then
\begin{equation}\label{star}
K_Y+T+B=f^*(K_X+S)+dE,
\end{equation}
and
$$\bfig
 \square/->`>`<-_{)}`>/<1300,500>[H^0(Y,m(K_Y+T+B))`H^0(T,m(K_T+dP))
 `H^0(X,m(K_X+S))`H^0(T,mK_T));r`\simeq`\neq`]
 \efig$$
In general, $r$ is not surjective. By \eqref{star},
$E\subset\Bs|m(K_Y+T+B)|$ and the assumption (2) of the Lifting Lemma is not satisfied.
\end{exa}
\begin{dfnlm}
Let $(X,\Delta)$ be a log pair. There is a b-divisor $\B=\B(X,
\Delta)$ on $X$ such that for all models $f\colon Y\rightarrow X$, we can write uniquely
$$K_Y+B_Y=f^*(K_X+\Delta)+E_Y,$$
where $B_Y$ and $E_Y$ are effective with no common components and $E_Y$ is $f$-exceptional.
We call $\B$ the {\em boundary\/}
b-divisor of the pair $(X,\Delta)$.
\end{dfnlm}
\begin{proof}
Let $g\colon Y'\rightarrow X$
be a model such that there is a proper birational morphism $h\colon
Y'\rightarrow Y$. Pushing forward
$K_{Y'}+B_{Y'}=g^*(K_X+\Delta)+E_{Y'}$ via $h_*$ yields
$$K_Y+h_*B_{Y'}=f^*(K_X+\Delta)+h_*E_{Y'},$$
and since $h_*B_{Y'}$ and $h_*E_{Y'}$ have no common components, $h_*B_{Y'}=B_Y$.
\end{proof}
\begin{lem}\label{disjoint}
Let $(X,\Delta)$ be a log canonical pair. There exists a log
resolution $Y\rightarrow X$ such that the
components of $\{\B_Y\}$ are disjoint.
\end{lem}
\begin{proof}
See \cite[Proposition 2.36]{KM98} or \cite[Lemma 6.7]{HM05}.
\end{proof}
\begin{cor}
If $(X,\Delta)$ is a klt pair, there are only finitely many valuations $E$ with $a(E,K_X+\Delta)<0$.
\end{cor}
\begin{proof}
Follows from Lemma \ref{disjoint} and \cite[Proposition 2.31]{KM98}.
\end{proof}
For any two divisors $P=\sum p_iE_i$ and $Q=\sum q_iE_i$ set
$$P\wedge Q=\sum\min\{p_i,q_i\}E_i.$$
\begin{dfn}\label{dfn:3}
Let $(X,\Delta)$ be a pair such that $I(K_X+\Delta)$ is Cartier and
the linear system $|I(K_X+\Delta)|$ is not empty. For a
log resolution $f\colon Y\rightarrow X$, let
$$mI(K_Y+\B_Y)=M_{mY}+F_{mY}$$
be the decomposition into mobile and fixed parts and set
$$B_{mY}=\B_Y-\B_Y\wedge(F_{mY}/mI).$$
Then
$$\Mob(mI(K_Y+B_{mY}))=M_{mY}=\Mob(mIf^*(K_X+\Delta)).$$
\end{dfn}
\begin{thm}\label{sequences}
Let $(X,\Delta)$ be a log canonical pair such that $I(K_X+\Delta)$ is
Cartier and the linear system $|I(K_X+\Delta)|$ is not empty. Fix a log resolution $f\colon Y\rightarrow X$ such that the support of
$$\Fix|I(K_Y+\B_Y)|+\Exc f+\B_Y$$
is simple normal crossings. There
is an effective sequence $\B_{\bullet}$ of b-divisors on
$Y$ such that $\B_m \leq \B$ is the smallest b-divisor such that for
all models $f\colon Z \rightarrow Y$ we have 
\[ H^0(Z, mI(K_Z+\B_{mZ}))= H^0(Z, mI(K_Z+\B_Z)).\]
The b-divisor $\K_Y+\B_m$
is good on any log resolution $Y_m\rightarrow Y$ where $M_{mY_m}$ is free,
and the sequence $\B_\bullet$ is concave after truncation.
\end{thm}
\begin{proof}
{\em Step 1.} Fix $m$ and set $\B'=\B(Y,\B_Y)$. Let $g\colon
Z\rightarrow Y$ be a log resolution and denote $h=g\circ f$. Let
$$K_Z+\B'_Z=g^*(K_Y+\B_Y)+E'_Z,$$
and define $M'_{mZ},F'_{mZ}$ and $B'_{mZ}$ as in Definition \ref{dfn:3}.

Firstly, since $K_Z+\B'_Z=h^*(K_X+\Delta)+g^*E_Y+E'_Z$, we have $\B'_Z=\B_Z+E$
where $E$ is effective. Furthermore,
\begin{align*}
M'_{mZ}&=\Mob(mIg^*(K_Y+\B_Y))\\
&=\Mob(mIh^*(K_X+\Delta)+mIg^*E_Y)=M_{mZ},
\end{align*}
so that
$$mIE=mI(K_Z+\B'_Z)-mI(K_Z+\B_Z)=F'_{mZ}-F_{mZ}.$$
This implies $B_{mZ}=B'_{mZ}$. Now we have
$$mI(K_Y+\B_Y)=g_*(mI(K_Z+\B'_Z))=g_*M'_{mZ}+g_*F'_{mZ},$$
and since $g_*M'_{mZ}=g_*M_{mZ}=M_{mY}$, we obtain $g_*F'_{mZ}=F_{mY}$. Therefore
\begin{align*}
g_*B_{mZ}&=g_*B'_{mZ}=g_*\B'_Z-g_*\big(\B'_Z\wedge(F'_{mZ}/mI)\big)\\
&=\B_Y-\B_Y\wedge(F_{mY}/mI)=B_{mY},
\end{align*}
and $\B_m$ is a well-defined b-divisor.\\[2mm]
\noindent
{\em Step 2.} Fix $Y_m$ as in the assumptions of the Theorem, and let $Z_m$ be another
log resolution such that $M_{mZ_m}$ is free. We may assume that there is a
birational morphism $g_m\colon Z_m\rightarrow Y_m$ such that the support of
$\Fix|g_m^*(K_{Y_m}+B_{mY_m})|+g_m^*B_{mY_m}+\Exc g_m$ has simple normal
crossings. Let $\Gamma_{mZ_m}= \B(Y_m, B_{mY_m})_{Z_m}$; then
\begin{equation}\label{nextdagger}
K_{Z_m}+\Gamma_{mZ_m}=g_m^*(K_{Y_m}+B_{mY_m})+E_{mZ_m},
\end{equation}
and $\Gamma_{mZ_m}\leq g_m^*B_{mY_m}$ since $Z_m$ and $Y_m$ are smooth.
Furthermore,
\begin{align*}
\Mob(mI(K_{Z_m}+\Gamma_{mZ_m}))&=\Mob(g_m^*(mI(K_{Y_m}+B_{mY_m})))\\
&=g_m^*M_{mY_m}=M_{mZ_m},
\end{align*}
and thus $\Fix|mI(K_{Z_m}+\Gamma_{mZ_m})|=g_m^*\Fix|mI(K_{Y_m}+B_{mY_m})|+E_{mZ_m}$.\\[2mm]
\noindent
{\em Step 3.}
Let $F$ be a common component of $g_m^*B_{mY_m}$ and
$g_m^*\Fix|mI(K_{Y_m}+B_{mY_m})|$. We claim that $a(F,Y_m,B_{mY_m})\geq0$, so that $F$ is a component
of $E_{mZ_m}$.

To this end, note that the centre of $F$ on $Y_m$ is contained in the intersection of
components of $B_{mY_m}$ and $\Fix|mI(K_{Y_m}+B_{mY_m})|$.
We can compute the discrepancy of $F$ on any model.
By \cite[Lemma 2.45]{KM98}, there is a composite $\rho_n\colon\mcal{W}_n\rightarrow Y_m$
of $n$ blow ups of the centres of $F$ such that $F$ is a divisor on $\mcal{W}_n$.
Let $\sigma\colon\mcal W_n\rightarrow\mcal W_{n-1}$ be the last blowup and let $\rho_{n-1}\colon\mcal W_{n-1}\rightarrow Y_m$ be
the composite of the first $n-1$ blowups; we can assume $c=\codim_{\mcal W_{n-1}}\sigma(F)>1$. Write
$$K_{\mcal W_{n-1}}+B_{\mcal W_{n-1}}^+=\rho_{n-1}^*(K_{Y_m}+B_{mY_m})+B_{\mcal W_{n-1}}^-,$$
where $B_{\mcal W_{n-1}}^+$ and $B_{\mcal W_{n-1}}^-$ are effective divisors without common components. Then by \cite[Lemmas 2.27, 2.30]{KM98},
\begin{align}
a(F,Y_m,B_{mY_m})&=a(F,W_{n-1},B_{W_{n-1}}^+-B_{W_{n-1}}^-)\label{eq:3}\\
&\geq a(F,W_{n-1},B_{W_{n-1}}^+),\notag
\end{align}
and by induction on $n$, $B_{W_{n-1}}^+$ and $\Fix|\rho_{n-1}^*(mI(K_{Y_m}+B_{mY_m}))|$ have no common components. But
$\sigma(F)\subset\Supp\Fix|\rho_{n-1}^*(mI(K_{Y_m}+B_{mY_m}))|$, so by the simple normal crossings assumption there can be at most $c-1$
components of $B_{\mcal W_{n-1}}^+$ that contain $\sigma(F)$. Denote these components by $\Delta_1,\dots,\Delta_p$, where $p\leq c-1$. Then
by \cite[Lemma 2.29]{KM98} and since $(\mcal W_{n-1},B_{\mcal W_{n-1}}^+)$ is log canonical, we have
\begin{equation}\label{eq:4}
a(F,\mcal W_{n-1},B_{\mcal W_{n-1}}^+)=c-1+\sum_{i=1}^pa(\Delta_i,\mcal W_{n-1},B_{\mcal W_{n-1}}^+)\geq0,
\end{equation}
and thus \eqref{eq:3} and \eqref{eq:4} prove the claim.

As a result, $\Gamma_{mZ_m}$ and $g_m^*\Fix|mI(K_{Y_m}+B_{mY_m})|$ have no
common components and thus $\Gamma_{mZ_m}=B_{mZ_m}$. This proves that $\K_Y+\B_m$ is good on $Y_m$.\\[2mm]
\noindent
{\em Step 4.}
To prove that $\B_\bullet$ is concave, observe that the sequence of effective
divisors $F_{mY}/mI$ is convex. For each component $B_i$ of $\B_Y$,
set $k_i=1$ if $\mult_{B_i}\B_Y\leq\mult_{B_i}F_{mY}/mI$ for all
$m\geq1$;  otherwise, since the sequence $F_{mY}/mI$ has a limit, set $k_i$
to be a positive integer such that
$\mult_{B_i}\B_Y\geq\mult_{B_i}F_{mY}/mI$ for all $m\geq k_i$.
Let $k=\max\{k_i\}$; the sequence
$\B_{km,Y}$ is concave.
\end{proof}
\begin{nt}
\label{nt1}
The sequence $\NN_m=mI(\K_Y+\B_m)$, the associated sequence of mobile b-divisors $\M_m$
(see Lemma \ref{mobile}) and the {\em characteristic\/} sequence $\D_m=\M_m/m$ are defined on $Y$.
We denote by $R(X,\NN_\bullet)$, $R(X,\M_\bullet)$ and $R(X,\D_\bullet)$ the same algebra.
\end{nt}
The following two theorems establish that a truncation of the
restricted algebra is an adjoint algebra.
\begin{thm}\label{sequences2}
Let $(X,\Delta)$ be a plt pair such that $I(K_X+\Delta)$ is
Cartier, the linear system $|I(K_X+\Delta)|$ is not empty, $S=\lfloor\Delta\rfloor$ is irreducible and is not a component of $\Fix|I(K_X+\Delta)|$.
Fix a log resolution $f\colon Y\rightarrow X$ such that the components of $\{\B_Y\}$ are disjoint and the support of
$$\Fix|I(K_Y+\B_Y)|+\Exc f+\B_Y$$
is simple normal crossings.
Denote $T=\widehat S_Y$ and $B=\B_Y-T$. 

Then there are b-divisors $\B^0_m$ on $T$ such that $\K_T+\B^0_m$ is
good on $T$, the limit $\B^0=\lim_{m\rightarrow\infty}\B^0_m$ exists as a b-divisor and the sequence
$\B^0_\bullet$ is concave after truncation.
More precisely, let $ Y_m\rightarrow Y$ be any log
resolution such that $M_{mY_m}$ is basepoint free and set $R_m=\widehat
S_{Y_m}$. Then $B^0_{mR_m}=(\B_{mY_m}-R_m)_{|R_m}$.
\end{thm}
\begin{rem}
We have the following diagram:
$$\bfig
 \square/->`>`<-_{)}`>/<1300,500>[H^0(Y,mI(K_Y+\B_{mY}))`H^0(T,mI(K_T+\B^0_{mT}))
 `H^0(X,mI(K_X+\Delta))`\res_S H^0(X,mI(K_X+\Delta));r`\simeq``]
 \efig$$
\end{rem}
\begin{proof}
{\em Step 1.} We follow the notation set in the proof of Theorem \ref{sequences}. Restricting
$$K_{Y_m}+\B'_{Y_m}=\pi_m^*(K_Y+\B_Y)+E'_{Y_m}$$
to $R_m$ gives
$$K_{R_m}+(\B'_{Y_m}-R_m)_{|R_m}=(\pi_{m|R_m})^*(K_T+B_{|T})+E'_{Y_m|R_m}.$$
Since the components of $B_{|T}$ do not intersect, the pair
$(T,B_{|T})$ is terminal by \cite[Proposition 2.31]{KM98},
thus $(\B'_{Y_m}-R_m)_{|R_m}=(\pi_{m|R_m})_*^{-1}B_{|T}$ and $E'_{Y_m|R_m}$
is exceptional. Therefore, the components of $B^0_{mR_m}$ do not
intersect and $(R_m,B^0_{mR_m})$ is terminal.

Let $Q_m=\widehat S_{Z_m}$ and recall from Step 3 of the proof of Theorem \ref{sequences} that
$\Gamma_{mZ_m}=B_{mZ_m}$, so that the restriction of \eqref{nextdagger} to $Q_m$ yields
$$K_{Q_m}+B^0_{mQ_m}=(g_{m|Q_m})^*(K_{R_m}+B^0_{mR_m})+E_{mZ_m|Q_m},$$
and thus $(g_{m|Q_m})_*^{-1}B^0_{mR_m}=B^0_{mQ_m}$ and $E_{mZ_m|Q_m}$ is exceptional.
Set $\B^0_m=\widehat{B^0_{mT}}$, where
$B^0_{mT}=(\pi_{m|R_m})_*B^0_{mR_m}\leq B_{|T}$. The pair $(T,B^0_{mT})$ is
terminal, and $\K_T+\B^0_m$ is good on $T$.\\[2mm]
\noindent
{\em Step 2.}
To prove concavity of $\B^0_\bullet$ observe that the normal crossings assumption on $Y_m$ ensures that
$$B^0_{mT}=B_{|T}-B_{|T}\wedge\big((\pi_{m|R_m})_*F_{mY_m|R_m}/mI\big).$$
Fix positive integers $i$ and $j$, and assume that $W=Y_i=Y_j=Y_{i+j}$.
Denote $V=\widehat S_W$ and $\mu\colon V\rightarrow T$. Then
$F_{iW}+F_{jW}\geq F_{i+j,W}$ implies
$$\mu_*F_{iW|V}+\mu_*F_{jW|V}\geq\mu_*F_{i+j,W|V},$$
and the sequence $(\pi_{m|R_m})_*F_{mY_m|R_m}/mI$ is convex.
The sequence $B^0_{km,T}$ is concave for some $k$ as in
Step 4 of the proof of Theorem \ref{sequences}. Therefore the limit
$\B^0$ exists since $\B^0_{mT}\leq B_{|T}$.
\end{proof}
\begin{nt}
\label{nt2}
On $T$, we consider the sequences $\NN_m^0=mI(\K_T+\B^0_m)$, the
associated sequence of mobile b-divisors $\M_m^0$ as in Lemma
\ref{mobile} and the {\em characteristic\/} sequence $\D_m^0=\M_m^0/m$.
\end{nt}
\begin{thm}\label{restrictedadjoint}
Under assumptions of Setup \ref{setup}, a truncation of the restricted algebra $R_S(X,K_X+\Delta)$ is an adjoint algebra.
More precisely, there is a log resolution $f\colon Y\rightarrow X$
such that, if $T=\widehat S_Y$, there is a natural sequence of b-divisors $\B^0_\bullet$ on $T$ with
$$R_S(X,K_X+\Delta)^{(s)}\simeq\bigoplus_{n\geq0}H^0(T,nsI(K_T+\B^0_{ns,T}))$$
for some positive integers $s$ and $I$.
\end{thm}
\begin{proof}
Every divisor on $X$ is big because $\pi$ is birational and by Kodaira's lemma
$B\sim_{\Q}A+C$, where $A$ is ample and $C\geq0$.
Since $\pi$ is small, every divisor on $X$ is mobile and we can assume that $S$ is not a component of $C$.
For $0<\varepsilon<1$ we have
$$B\sim_{\Q}(1-\varepsilon)B+\varepsilon C+\varepsilon A,$$
and $(X,(1-\varepsilon)B+\varepsilon C+\varepsilon A)$ is klt for $\varepsilon\ll1$.
Replacing $A$ by $\varepsilon A$ and $C$ by
$(1-\varepsilon)B+\varepsilon C$ we may assume that $B=A+C$ with $A$ ample and $C$ effective, and
such that $S\not\subset\Supp C$.
Furthermore, if $k$ is an integer such that $kA$ is very ample, we may
assume $kA$ is a very general member of the linear system $|kA|$ so that $A$
is transverse to all other relevant divisors.

Let $f\colon Y\rightarrow X$ be a resolution as in Theorem \ref{sequences2}.
Observe that $R^{(I)}=R(Y,\NN_\bullet)$, and we will show that $R_S^{(I)}=R(T,\NN^0_\bullet)$.

Fix a positive integer $m$. It will be enough to apply the Lifting Lemma to conclude that the restriction map
$H^0(Y,\NN_m)\rightarrow H^0(T,\NN_m^0)$ is surjective.

For condition (1) of the Lifting Lemma, since $kA$ is very general in
$|kA|$, $f^*A=f_*^{-1}A$ is free and hence $\B_{mY}-T\geq f^*A$.
We can choose a small $f$-exceptional effective divisor $E$ on $Y$ such that $A'=f^*A-E$ is ample, and
$\B_{mY}-T=A'+C'$, where the support of $C'=\B_{mY}-T-f^*A+E$ does not contain $T$.

We prove that condition (2) of the Lifting Lemma is satisfied.
We can assume that $\Mob\NN_{mY}$ is free.
Since $\Fix|\NN_{mY}|$ and $\B_{mY}$ have simple normal crossings supports and
have no common components, $\Bs|\NN_{mY}|$ does not contain any
intersection of components of $\B_{mY}$. This finishes the proof.
\end{proof}

\begin{cor}\label{restriction}
$\M_i^0={\M_i}_{|T}$ for every $i$.
\end{cor}
\begin{proof}
Fix $i$. Let $Y'\rightarrow Y$ be a resolution such that, if
$W=\widehat T_{Y'}$, $\M_i$ descends to $Y'$ and $\M_i^0$ descends to $W$.
Restricting $\M_{iY'}\leq\NN_{iY'}$ to $W$ and taking mobile parts
shows that $\M_{iY'|W}\leq\M_{iW}^0$,
and hence
$$H^0(W,\M_{iY'|W})\subset H^0(W,\M_{iW}^0).$$
As
$$H^0(W,\M_{iW}^0)=\Image\big(H^0(Y',\M_{iY'})\rightarrow
H^0(W,\M_{iY'|W})\big),$$ $\M_{iY'|W}=\M_{iW}^0$ because both divisors are free.
\end{proof}

\begin{thm}\label{corltmodels}
Assume existence and finiteness of models. Let $(X,\Delta)$ be a $\Q$-factorial klt pair projective over $Z$ with $K_X+\Delta$ big, and let $V\subset\Div(X)_\R$ be a
finite-dimensional vector space that contains $\Delta$. Let $\Delta_\bullet$ be a
concave sequence with $\lim_{m\rightarrow\infty}\Delta_m=\Delta$ and assume there is a
positive integer $I$ such that $I(K_X+\Delta_m)$ is integral for every $m$.
Then there exist a smooth model $f\colon Y\rightarrow X$ and a positive integer $s$ such that
$$\lim_{m\rightarrow\infty}\Mob(msIf^*(K_X+\Delta_m))/m$$
exists and is a semiample divisor.
\end{thm}
\begin{proof}
By Kodaira's Lemma, $K_X+\Delta \sim_{\R,f} A+C$, where
$A$ is ample and $C$ is effective. If $\Delta'= \varepsilon A+\varepsilon C+\Delta$ for $0<\varepsilon\ll1$,
the pair $(X,\Delta')$ is klt and $K_X + \Delta' \sim_{\R,f} (1 + \varepsilon)(K_X + \Delta)$.
Therefore replacing $\Delta$ by $\Delta'$ we may assume that $\Delta$ is big.

Let $\varphi_i\colon X\dashrightarrow W_i$ be finitely many weak log canonical models corresponding to the polytope around $\Delta$.
Let $f\colon Y\rightarrow X$ be
a common log resolution of the pair $(X,\Delta)$ and of the pairs
$(W_i,(\varphi_i)_*\Delta)$ for all $i$, and denote $\pi_i\colon Y\rightarrow W_i$ the
induced maps.
For every $m\gg0$ there is an index $i_m$ such that
$K_{W_{i_m}}+(\varphi_{i_m})_*\Delta_m$ is semiample by Theorem \ref{thm:3}.

Since $W_i$ have rational singularities, by the proof of \cite[Lemma
1.1]{Kaw88}, the groups $\WDiv(W_i)/\Div(W_i)$ are finitely
generated, and hence finite because each $W_i$ is $\Q$-factorial.
Therefore, for $m\gg0$ there are positive integers $s_{i_m}$ such that
$s_{i_m}I(K_{W_{i_m}}+(\varphi_{i_m})_*\Delta_m)$ are Cartier.
By Effective Basepoint Freeness, there is a positive integer $c$ such
that all $cs_{i_m}I(K_{W_{i_m}}+(\varphi_{i_m})_*\Delta_m)$ are free;
set $s=c\prod s_i$.

By the Negativity Lemma, for $m\gg0$ we can write
$$f^*(K_X+\Delta_m)=\pi_{i_m}^*(K_{W_{i_m}}+(\varphi_{i_m})_*\Delta_m)+E_m,$$
where $E_m$ are effective and exceptional.
There is an index $i$ such that $i=i_m$ for infinitely many $m$ and therefore the divisor
\begin{align*}
\lim_{m\rightarrow\infty}\Mob(msIf^*(K_X+\Delta_m))/m&=\lim_{m\rightarrow\infty}sI\pi_i^*(K_{W_i}+(\varphi_i)_*\Delta_m)\\
&=sI\pi_i^*(K_{W_i}+(\varphi_i)_*\Delta)
\end{align*}
is nef and thus semiample by Theorem \ref{thm:3}.
\end{proof}
\begin{dfn}
An adjoint algebra $R(X,\D_\bullet)$ is {\em semiample\/} if
$\D=\lim_{i\rightarrow\infty}\D_i$ exists as a real mobile b-divisor.
\end{dfn}

\begin{thm}\label{semiample}
Let $(X,\Delta)\rightarrow Z$ be a pl flipping contraction as in Setup \ref{setup}, where $\dim
X=n$. Assume existence and finiteness of models in
dimension $n-1$. Then there is a positive integer $s$ such that $R_S(X,K_X+\Delta)^{(s)}$
is a semiample adjoint algebra. Moreover, we can assume that there is a model $T$ to which all $\M_m^0$ descend.
\end{thm}
\begin{proof}
Let a log resolution $f\colon Y\rightarrow X$ and the associated
sequence $\B^0_\bullet$ be as in Theorem~\ref{sequences2}.
Let $\B^0=\lim\limits_{i\rightarrow\infty}\B^0_i\in\bDiv(T)_\R$. By Corollary \ref{corltmodels} there is a model $W\rightarrow T$
and a positive integer $s$ such that $\M^0_{ms,W}$ is free for every $m$ and $\lim\limits_{m\rightarrow\infty}\M^0_{ms,W}/m$
is semiample. We are done.
\end{proof}
\section{Shokurov's saturation condition}

\begin{dfn}
Let $(X,\Delta)$ be a log pair, let $\A=\A(X,\Delta)$ and let $\FF$ be a b-divisor such that $\lceil \FF\rceil\geq0$.
\begin{itemize}
\item An integral mobile b-divisor $\M$ on $X$ is {\em $\FF$-saturated\/} if there is a model $Y\rightarrow X$ such that
$$\Mob\lceil \M_W+\FF_W\rceil\leq \M_W$$
on all models $W\rightarrow Y$. If $(X,\Delta)$ is klt and $\FF=\A$ then we say that $\M$ is saturated.
\item An adjoint algebra $R=R(X,\NN_\bullet)$ is (asymptotically) $\FF$-sa\-tu\-ra\-ted if there is a model
$Y\rightarrow X$ such that on every log smooth model $(W,\FF_W)$ over $Y$, for all $i\geq j>0$ we have
$$\Mob\lceil j\D_{iW}+\FF_W\rceil\leq j\D_{jW}.$$
If $(X,\Delta)$ is klt and $\FF=\A$ then we say that $R$ is saturated.
\end{itemize}
\end{dfn}

\begin{rem}
Any truncation of a saturated adjoint algebra is saturated.
\end{rem}
\begin{exa}
If $M$ is a divisor on a klt pair $(X, \Delta)$, then $\M=\overline M$ is saturated.
Indeed, for every model $f\colon Y\rightarrow X$ we have
$f_*\OO_Y(f^*M+E)=\OO_X(M)$ for an effective and $f$-exceptional divisor $E$.
\end{exa}

We prove that saturated semiample adjoint algebras on curves re finitely generated; this models the proof in higher dimensions.

\begin{pro}
Let $X=\mbb{A}_{\C}^1$. A saturated adjoint algebra on $X$ is finitely generated.
\end{pro}
\begin{proof}
On a curve b-divisors are just ordinary divisors.
Consider an adjoint algebra $R=R(X,N_\bullet)$. Since on an affine
curve all linear systems are basepoint free, we have $N_i=M_i=iD_i$ and
since $A=-\Delta$, the saturation condition reads
$$\lceil jD_i-\Delta\rceil\leq jD_j.$$
This is a componentwise condition, so we can assume that all divisors
are supported at a point $P\in\mbb{A}_{\C}^1$. Let $\Delta=bP$ with $0<b<1$
and $D_i=d_iP$, so that
$$R(X,N_\bullet)=\bigoplus_{i\geq0}H^0(\mbb{A}_{\C}^1,\OO_{\mbb{A}_{\C}^1}(id_iP)).$$
The sequence $d_i$ satisfies $d_i\geq0$ and
$d_{i+j}\geq\frac{i}{i+j}d_i+\frac{j}{i+j}d_j$, so it has a limit
$d=\sup d_i$.
The saturation condition when $i \rightarrow \infty$ becomes
$$\lceil jd-b\rceil\leq jd_j$$
for all $j>0$. Denote $q=\lfloor1/(1-b)\rfloor$.
\begin{cla}
$q!d\in\N$ and $d=d_j$ if $jd\in\N$.
\end{cla}
In particular if $d=u/v$ with $u,v\in\N$, then $R^{(v)}=R(\mbb{A}_{\C}^1,uP)$ and $R$ is finitely generated.

To prove the claim, first assume that $d\notin\Q$; then there is
$j\in\N$ such that $\{jd\}>b$ and therefore
$$jd_j\leq jd<\lceil jd-b\rceil\leq jd_j,$$
a contradiction. Assume now that $jd\in\Z$, then
$$jd_j\geq\lceil jd-b\rceil=jd\geq jd_j,$$
and therefore $d=d_j$. Choose the smallest such $j$ and set $l=jd$. Then $l$ and $j$ are coprime, so there is $k<j$ such that $kl\equiv-1\pmod{j}$.
But then if $j>q$, $\{kd\}=\{kl/j\}=(j-1)/j>b$ and this is a contradiction as above.
\end{proof}

\begin{lem}\label{sat}
Let $Y\rightarrow X$ be a log resolution of a pair $(X,\Delta)$.
Assume that $\NN$ is a real b-divisor on $X$
which descends to $Y$ and such that $\Supp\NN_Y$ is simple normal crossings. Then
$\lceil\NN+\A\rceil$ is good on $Y$.
\end{lem}
\begin{proof}
Let $f\colon Y'\rightarrow Y$ be a model and
let $G$ be any simple normal crossings divisor on $Y$. The pair $(Y,\lceil G\rceil-G)$ is klt and in particular the divisor
$$E=\lceil K_{Y'}-f^*(K_Y+\lceil G\rceil-G)\rceil$$
is effective and exceptional. Thus $f^*\lceil G\rceil+E=\lceil K_{Y'/Y}+f^*G\rceil$. Consider the divisor
$G=\NN_Y+\A_Y$ and observe that $\A_{Y'}=f^*\A_Y+K_{Y'/Y}$; this yields:
$$f^*\lceil\NN_Y+\A_Y\rceil+E=\lceil f^*\NN_Y+\A_{Y'}\rceil=\lceil\NN_{Y'}+\A_{Y'}\rceil.$$
\end{proof}

\begin{lem}\label{blowups}
Let $S$ be a subvariety of a projective variety $X$ and let $h\colon
T\rightarrow S$ be a projective birational map.
Then there is a projective birational map $f\colon Y\rightarrow X$
from a variety $Y$ and a closed immersion $j\colon T\rightarrow Y$
such that $f\circ j=h$.

Furthermore, assume that $h$ is proper but not necessarily projective. Then there
exist a projective birational map $f\colon Y\rightarrow X$
and a subvariety $T'\subset Y$ such that $f_{|T'}$ factors through $h$.
\end{lem}
\begin{proof}
Let $i\colon S\rightarrow X$ be the closed immersion. By \cite[Chapter
II, Theorem 7.17]{Har77} there is a coherent ideal sheaf $\mcal J$ on $S$
such that $T$ is the blowup of $S$ with respect to $\mcal J$. Let $Y$
be the blowup of $X$ with respect to the ideal sheaf $i_*\mcal J$.
The conclusion follows from \cite[Chapter II, Corollary
7.15]{Har77} since $i^{-1}(i_*\mcal J)\cdot\OO_S=\mcal J$.

The second statement follows from Chow's lemma.
\end{proof}

\begin{thm}
\label{sat2}
Let $(X,\Delta)\rightarrow Z$ be a pl flipping contraction as in Setup \ref{setup}, where $\dim X=n$. Assume the MMP with scaling
for $\Q$-factorial klt pairs with big boundary in dimension $n-1$. There is a positive integer $s$ such that $R_S(X,K_X+\Delta)^{(s)}$
is saturated.
\end{thm}
\begin{proof}
Let $s$ and $T$ be a positive integer and a model as in Theorem
\ref{semiample}; we assume that $T$ sits on a model $f\colon Y\rightarrow X$
by Lemma \ref{blowups}. Up to truncation we may assume that $s=1$ and
it is enough to show that the saturation condition holds on $T$.
Set $\A^*=\widehat S+\A$ and $\A^0=\A(S,B_S)$.

Fix integers $i\geq j>0$.
Let $\varphi\colon Y'\rightarrow Y$ be a log resolution such that $\M_i$ and $\M_j$ descend to $Y'$
and $\M_i^0$ and $\M_j^0$ descend to $W=\widehat T_{Y'}$. Furthermore, by Lemma \ref{sat} and Lemma \ref{mobile} we can assume
that $\Mob\lceil j\D_{iY'}+\A_{Y'}^*\rceil$ is free.
Let $g=f\circ\varphi$ and observe that $\lceil\A_{Y'}^*\rceil$ is effective and $g$-exceptional.
Denoting $M=I(K_X+S+B)$, since $\D_{iY'}=(1/i)\Mob g^*(iM)$ we have
\begin{align*}
\Mob\lceil j\D_{iY'}+\A_{Y'}^*\rceil&\leq\Mob\lceil(j/i)g^*(iM)+\A_{Y'}^*\rceil\\
&=\Mob g^*(jM)=j\D_{jY'},
\end{align*}
and therefore
$$(\Mob\lceil j\D_{iY'}+\A_{Y'}^*\rceil)_{|W}\leq j\D_{jW}^0,$$
since $\D_{iW}^0=\D_{iY'|W}$ the proof of Corollary \ref{restriction}.

By adjunction $\A^*_{Y'|W}=\A_W^0$. Taking global sections in the sequence
\begin{multline*}
0\rightarrow\OO_{Y'}(\lceil j\D_{iY'}+\A_{Y'}\rceil)\\
\rightarrow\OO_{Y'}(\lceil j\D_{iY'}+\A_{Y'}^*\rceil)
\rightarrow\OO_W(\lceil j\D_{iW}^0+\A_W^0\rceil)\rightarrow0
\end{multline*}
we obtain that the restriction map
$$H^0(Y',\lceil j\D_{iY'}+\A_{Y'}^*\rceil)\rightarrow H^0(W,\lceil j\D_{iW}^0+\A_W^0\rceil)$$
is surjective since $H^1(Y',\lceil j\D_{iY'}+\A_{Y'}\rceil)=(0)$ by Kawamata-Viehweg vanishing.
Then $(\Mob\lceil j\D_{iY'}+\A_{Y'}^*\rceil)_{|W}=\Mob\lceil j\D_{iW}^0+\A_W^0\rceil$ as in the proof of
Corollary \ref{restriction}, and therefore
$$\Mob\lceil j\D_{iW}^0+\A_W^0\rceil\leq j\D_{jW}^0.$$
Finally by Lemma \ref{sat} and Lemma \ref{mobile},
\begin{align*}
\Mob\lceil j\D_{iT}^0+\A_T^0\rceil&=(g_{|W})_*\Mob\lceil j\D_{iW}^0+\A_W^0\rceil\\
&\leq(g_{|W})_*j\D_{jW}^0=j\D_{jT}^0.
\end{align*}
This concludes the proof.
\end{proof}
\section{Existence of flips}

\begin{applem}
Let $Y$ be a nonsingular variety and let $P_1,\dots,P_l$  be free integral divisors on $Y$.
Let $D\in\sum\R_+P_k$ and assume that $D$ is not a $\Q$-divisor. Let $\sum G_i$ be a reduced divisor containing
$\Supp P_k$ for every $k$.
Then for every $0<\varepsilon\ll1$ there exists a free integral divisor $M$ and a positive integer $j$ such that:
\begin{enumerate}
\item $\|jD-M\|_\infty<\varepsilon$, where the sup-norm is taken with respect to the basis $\{G_i\}$ of the vector space
$\bigoplus\R G_i\subset\Div(Y)$,
\item $jD-M$ is not effective.
\end{enumerate}
\end{applem}
\begin{proof}
Set $L_\Z=\bigoplus\Z P_k\simeq\Z^l$ and $L_\R=\bigoplus\R P_k\simeq\R^l$, and assume that $D$ is represented in that basis by
$\mathbf{d}=(d_1,\dots,d_l)\in\R_+^l$, where not all $d_i$ are rational.

Let $\pi\colon L_\R\rightarrow L_\R/L_\Z$ be the projection map and let $A=\N\pi(\mathbf{d})\subset L_\R/L_\Z$.
The closure $\overline A$ is a positive dimensional Lie subgroup; denote by $\overline A_0$  the positive dimensional real sub-torus
that is the connected component of $0$. Then $T=\pi^{-1}(\overline A_0)\subset L_\R$ is a positive dimensional vector space and $T$ is not a subset of the cone
$L_\R\cap\sum\R_+G_i$. Therefore for every $\varepsilon>0$ there is a positive integer $j$ and an element $\mathbf{m}=(m_1,\dots,m_l)\in\Z^l$
such that $\|j\mathbf{d}-\mathbf{m}\|_\infty<\varepsilon$ and $j\mathbf{d}-\mathbf{m}\notin L_\R\cap\sum\R_+G_i$. If $\varepsilon\ll1$, then
$\mathbf{m}\in\N^l$ and thus $\sum m_iP_i$ is free.
\end{proof}

Finally we can prove:

\begin{thm}
\label{fg6}
Let $R=R(Y,\NN_\bullet)$ be an adjoint algebra. If $R$ is saturated
and semiample then it is finitely generated.
\end{thm}
\begin{proof}
Firstly we prove that $\D=\lim\limits_{i\rightarrow\infty}\D_i$ is a $\Q$-b-divisor. Assume otherwise. Passing to a resolution we may assume that
$\D_Y$ is semiample. Letting $i\rightarrow\infty$ in the saturation condition, we have
$$\Mob\lceil j\D_Y+\A_Y\rceil\leq j\D_{jY}$$
for all $j>0$. Let $P_1,\dots,P_l$  be free integral divisors on $Y$ such that
$\D_Y\in\sum\R_+P_k$, and let $G=\sum G_i$ be a reduced divisor containing $\Supp P_k$ for every $k$.
Fix $\varepsilon>0$ small enough so that $\lceil\A_Y-\varepsilon G\rceil\geq0$. By the Approximation Lemma
there is a free integral divisor $M$ and a positive integer $j$ such that $\|j\D_Y-M\|_\infty<\varepsilon$ and $j\D_Y-M$ is not effective.
Then $j\D_Y-M\geq-\varepsilon G$ and
$$\lceil j\D_Y+\A_Y\rceil=\lceil M+\A_Y+(j\D_Y-M)\rceil\geq M+\lceil\A_Y-\varepsilon G\rceil\geq M,$$
so that
$$j\D_Y\geq j\D_{jY}\geq\Mob\lceil j\D_Y+\A_Y\rceil\geq M.$$
This is a contradiction.

Let now $j$ be a positive integer such that $j\D_Y$ is an integral free divisor. Then
$$j\D_{jY}\leq j\D_Y\leq\Mob(j\D_Y+\lceil\A_Y\rceil)\leq j\D_{jY},$$
and thus $\D_Y=\D_{jY}$ and $\D=\D_j$.
\end{proof}

\begin{cor}
Assume existence and finiteness of models in dimension $n-1$. Then klt flips exist in di\-men\-si\-on $n$.
\end{cor}
\begin{proof}
By Theorem~\ref{thm:41} and Remark~\ref{Sirred}, it is enough to prove that the flip of a pl flipping contraction
$(X,\Delta)\rightarrow Z$ exists, where $S=\lfloor\Delta\rfloor$ is irreducible.
Since the problem is local on the base we assume that $Z$ is affine, and therefore it is enough to prove that
the canonical algebra $R(X,K_X+\Delta)$ is finitely generated. By Theorems~\ref{restrictedadjoint}, \ref{semiample} and \ref{sat2},
a truncation of the restricted algebra $R_S(X,K_X+\Delta)$ is a saturated semiample adjoint algebra, and thus is finitely generated by Theorem~\ref{fg6}.
\end{proof}

\bibliography{biblio1}
\pagestyle{plain}
\end{document}